\documentclass[11pt,a4paper]{article}

\usepackage[final]{optional}

\usepackage{amsfonts, amsmath, wasysym}
\usepackage{stmaryrd} 
\usepackage{esint} 

\usepackage{tikz}
\usetikzlibrary{shapes, positioning}

\usepackage{amssymb,amsthm,
paralist
}

\usepackage{
latexsym,
}

\usepackage{url}

\definecolor{darkgreen}{rgb}{0,0.5,0}
\definecolor{darkred}{rgb}{0.7,0,0}
\usepackage[colorlinks, 
citecolor=darkgreen, linkcolor=darkred
]{hyperref}

\theoremstyle{plain}

\numberwithin{equation}{section}

\newcommand{\pl}[2]{{\frac{\partial #1}{\partial #2}}}

\newcommand{\pll}[2]{{\frac{\partial^2 #1}{\partial #2^2}}}

\newcommand{\ti}{\tilde}
\newcommand{\al}{\alpha}

\newcommand{\de}{\delta}

\newcommand{\vph}{\varphi}
\newcommand{\ep}{\varepsilon}

\newcommand{\R}{\ensuremath{{\mathbb R}}}
\newcommand{\N}{\ensuremath{{\mathbb N}}}

\newcommand{\weakto}{\rightharpoonup}
\newcommand{\downto}{\downarrow}
\newcommand{\upto}{\uparrow}

\newcommand{\lap}{\Delta}

\newcommand{\union}{\cup}

\DeclareMathOperator{\Vol}{Vol}

\def\blbox{\quad \vrule height7.5pt width4.17pt depth0pt}

\newcommand{\beq}{\begin{equation}}
\newcommand{\eeq}{\end{equation}}
\newcommand{\beqa}{\begin{equation}\begin{aligned}}
\newcommand{\eeqa}{\end{aligned}\end{equation}}
\newcommand{\brmk}{\begin{rmk}}
\newcommand{\ermk}{\end{rmk}}
\newcommand{\partref}[1]{\hbox{(\csname @roman\endcsname{\ref{#1}})}}
\newcommand{\half}{\frac{1}{2}}

\newcommand{\cmt}[1]{\opt{draft}{{\color[rgb]{0.5,0,0}
{$\LHD$ {#1} $\RHD$\marginpar{\blbox}}}}}

 \newtheorem{thm}{Theorem}[section]

\newtheorem{lem}[thm]{Lemma}

\newtheorem{rmk}[thm]{Remark}

\newcommand{\norm}[1]{\left\Vert#1\right\Vert}
\newcommand{\abs}[1]{\left\vert#1\right\vert}
\newcommand{\set}[1]{\left\{#1\right\}}
\newcommand{\Real}{\mathbb R}

\title{\sc uniqueness of ricci flows from nonatomic radon measures on riemann surfaces}
\author{Peter M. Topping and Hao Yin}
\date{14 June 2023}

\begin{document}

%

\parskip 8pt
\parindent 0pt

\maketitle

\begin{abstract}
In previous work \cite{TY3} we established the existence of a Ricci flow starting with a Riemann surface coupled with a nonatomic Radon measure as a conformal factor. 
In this paper we prove uniqueness, settling Conjecture 1.3 from \cite{TY3}. 
Combining these two works yields a canonical smoothing of 
such rough surfaces that also regularises their geometry at infinity.
\end{abstract}

\section{Introduction}

Suppose $M$ is a two-dimensional smooth manifold with a conformal structure. By assumption our manifolds are connected but not supposed to be compact, unless otherwise specified.
The conformal structure can be viewed as an equivalence class of Riemannian metrics 
where two metrics are deemed to be equivalent if they differ only by multiplication by a smooth function. The Ricci flow can be viewed as a way of evolving a representative $g$ of the conformal structure under the PDE
\beq
\label{RF_eq}
\pl{g}{t}=-2K_g g,
\eeq
where $K_g$ is the Gauss curvature of $g$. If we pick local isothermal coordinates $x$ and $y$, and write the metric $g=u(dx^2+dy^2)$, then the conformal factor $u$ evolves under the logarithmic fast diffusion equation
\beq\label{RF_eq_u}
\pl{u}{t}=\lap\log u.
\eeq
Here, and throughout this paper, $\lap:=\pll{}{x}+\pll{}{y}$ is the Laplacian with respect to the isothermal coordinates that we chose.

There is an extensive theory of Ricci flow in two dimensions that we briefly survey.
In the case that the underlying surface $M$ is closed, Hamilton and Chow 
\cite{ham_surf, chow}
proved existence and uniqueness of smooth solutions $g(t)$ on time intervals $[0,T)$ 
once a smooth initial metric $g(0)$ has been prescribed, and the optimal existence time $T$ was determined. A theory in which rough initial data was prescribed, instead of a smooth initial metric, was developed by Guedj and Zeriahi (existence \cite{GZ})
and Di Nezza and Lu (uniqueness \cite{DL}). Their theory was again constrained to closed manifolds, but applied to more general K\"ahler Ricci flows.

The nature of the well-posedness question changes completely once the compactness of the underlying manifold is no longer assumed. The focus changes to the problem of how Ricci flow might lose information at spatial infinity, or be influenced by information coming in from infinity. Intuition from classical PDE theory can be misleading because the behaviour of solutions is heavily influenced by the fact that the solution and the underlying geometry are the same thing.

The smooth existence problem on general surfaces (without any assumption of compactness) was solved by Giesen and the first author \cite{GT2}, following the programme in \cite{JEMS}. There it was proved that on a general smooth surface $M$, if we are given a smooth Riemannian metric $g_0$, which is not assumed to be complete and is not assumed to have bounded curvature, there exists a smooth Ricci flow $g(t)$ for $t$ in an optimal time interval $[0,T)$, with $g(0)=g_0$, and so that $g(t)$ is complete for all $t\in (0,T)$. This \emph{instantaneously complete} Ricci flow was \emph{maximally stretched} in the sense that if $\ti g(t)$, $t\in [0,\ti T)$, was any other smooth Ricci flow on $M$, with $\ti g(0)\leq g(0)$ and with $\ti g(0)$ and $g(0)$ in the same conformal class (but not assuming completeness of $\ti g(t)$) then $\ti g(t)\leq g(t)$ for 
$t\in [0,T\wedge \ti T)$.

This existence theory was generalised to rough initial data by the present authors in 
\cite{TY3}. The essential idea there was that instead of prescribing the initial data as a smooth Riemannian surface, which could be viewed equivalently as a manifold with a conformal structure and a smooth function as a conformal factor, we replaced the smooth conformal factor by a Radon measure $\mu$ on $M$ with the property that 
$\mu(\{x\})=0$ for all points $x\in M$. Once that relaxation had been made, it was necessary to weaken what it meant to satisfy the initial data. We say that a Ricci flow $g(t)$, $t\in (0,T)$, attains $\mu$ as initial data (weakly) if the volume measures $\mu_{g(t)}$ converge weakly to $\mu$, written $\mu_{g(t)}\weakto \mu$.
This means that for all $\psi\in C_c^0(M)$ we have
\beq
\label{weak_sense}
\int_M \psi d\mu_{g(t)} \to \int_M \psi d\mu,\quad\text{ as }t\downto 0.
\eeq
This is weak-* convergence when viewed in the dual of 
$C_c^0(M)$.
In \cite{TY3}, we proved existence of such a Ricci flow on an optimal time interval
that we will incorporate into our main theorem \ref{main_exist_thm} below.

This discussion addresses the existence problem in full generality, but says nothing about the problem of uniqueness that is so important in applications, except on closed manifolds as discussed above. In general dimension if one works in the class of smooth  solutions and initial data that is complete and of uniformly bounded curvature then uniqueness was proved by Chen and Zhu \cite{chenzhu}, with a simplified  proof given by 
Kotschwar \cite{kotschwar}. Once the hypothesis of bounded curvature is fully dropped, the problem becomes harder. In his influential paper \cite{strong_uniqueness},
B.-L Chen proved that the stationary solution is the only smooth complete Ricci flow starting with  Euclidean space in two or three dimensions. An audacious conjecture that is not currently generally accepted would be that any two complete Ricci flows $g_i(t)$, $t\in [0,T_i)$, for $i=1,2$, on the same underlying manifold and with the same initial metric $g_1(0)=g_2(0)$ must agree for $t\in [0,T_1\wedge T_2)$, in any dimension.
This conjecture was proved by the first author for two-dimensional $M$ in \cite{ICRF_UNIQ} but the question in general dimension is wide open.

In this paper we return to two-dimensional $M$ and consider 
the uniqueness problem  with initial data  a nonatomic Radon measure $\mu$ attained as described in the weak sense \eqref{weak_sense} above.
Even for \emph{smooth} initial data attained in a weak sense the uniqueness problem is far from the smooth theory.
As an illustration of this, in our previous work \cite[Theorem 1.5]{TY3}
we  proved nonuniqueness for Ricci flow starting with the Euclidean plane if the solution 
$g(t)$, $t\in (0,T)$, is asked to attain the initial data in the sense that the Riemannian distance converges  locally uniformly as $t\downto 0$.


Nevertheless, in this paper we solve the problem for Radon measure $\mu$ initial data by proving the uniqueness part of the following main theorem, which incorporates the existence from \cite{TY3}.
We use the notation $\ti M$ for the universal cover of $M$, and $\ti\mu$ for the corresponding lift of $\mu$ to $\ti M$. We write $D$ for the unit disc in the plane, equipped with the standard conformal structure.


\begin{thm}[Main existence and uniqueness theorem]
\label{main_exist_thm}
Let $M$ be a two-dimensional smooth manifold equipped with a conformal structure, 
and let $\mu$ be a Radon measure on $M$ that is nonatomic in the sense that
$$\mu(\{x\})=0\quad\text{ for all }x\in M.$$
Define $T\in [0,\infty]$ by
\begin{itemize}
\item
$T=\infty$ if $\ti M=D$;
\item
$T=\frac{1}{4\pi}\ti\mu(\ti M)$ if $\ti M=\R^2$;
\item
$T=\frac{1}{8\pi}\ti\mu(\ti M)$ if $\ti M=S^2$.
\end{itemize}
Then there exists a smooth complete conformal Ricci flow $g(t)$ on $M$, for $t\in (0,T)$,
attaining $\mu$ as initial data in the sense that 
$$\mu_{g(t)}\weakto \mu\text{ as }t\downto 0$$
and so that if $\ti g(t)$, $t\in (0,\ti T)$, is any other smooth complete conformal Ricci flow on $M$ that attains $\mu$ as initial data in the same sense, then 
$\ti T\leq T$ and 
$$g(t)\equiv \ti g(t)\qquad\text{ for all }t\in (0,\ti T).$$
If $T\in (0,\infty)$ then $\mu_{g(t)}(M)=(1-\frac{t}{T})\mu(M)$ for all $t\in (0,T)$.
\end{thm}

Note that even in the case that $\mu$ is smooth, i.e. is the volume measure of a smooth conformal metric on $M$, this theorem has been open until this point, and is not contained in the uniqueness theorem of \cite{ICRF_UNIQ}, because the initial data is only assumed to be attained by $\ti g(t)$ in the weak sense.


Observe that in the isolated cases in which $T=0$, the existence part of the theorem above is vacuous but the uniqueness is not. We must still prove nonexistence of an alternative solution.

The solution constructed in Theorem \ref{main_exist_thm} is the largest in the sense described in the following theorem.
\begin{thm}[Ordered initial data gives ordered Ricci flows]
\label{ordered_thm}
Let $M$, $\mu$ and $T$ be as in Theorem \ref{main_exist_thm}, and let 
$g(t)$ be the unique Ricci flow constructed in that theorem.
Suppose now that $\nu$ is any other Radon measure on $M$ with $\nu\leq \mu$,
and $\ti g(t)$, $t\in (0,\ti T)$, is any smooth conformal Ricci flow  on $M$ (possibly incomplete)
attaining $\nu$ as initial data. Then $\ti T\leq T$ and 
$$g(t)\geq \ti g(t)$$
for all $t\in (0,\ti T)$.
\end{thm}


This theorem gives the strongest assertion that the solution from 
Theorem \ref{main_exist_thm} is `maximally stretched' in the sense of \cite{GT2}.
Note that even in the case that $\nu=\mu$ and $\mu$ is smooth, this theorem has been open until this point, and is not contained in \cite{GT2}, again because the initial data is only assumed to be attained by $\ti g(t)$ in the weak sense.

The proof of  Theorem \ref{ordered_thm}, found in Section \ref{ordered_sect},
will invoke the existence and uniqueness of Theorem \ref{main_exist_thm}.
However, in order to prove Theorem \ref{main_exist_thm} itself 
(Section \ref{main_thm_pf_sect})
we will need to first prove a weaker `maximally stretched' assertion that we give in Theorem \ref{DC_max_stretch_thm}, 
and prove in Section \ref{max_stretch_sect},
in which we assume in particular that $\nu=\mu$.
As mentioned earlier, the special case of $M$ closed in Theorem \ref{main_exist_thm} is covered in previous work. As a service to the reader, we translate the uniqueness proof of Di Nezza and Lu \cite{DL} into the language of this paper in Appendix \ref{sec:s2}.

The theory in this paper has a number of immediate applications. 
It yields a complete analysis of expanding Ricci solitons in two dimensions, as will be explained in 
\cite{PT}. This application requires crucially both the ability to work with very rough initial data and the ability to work on noncompact manifolds.
Several of the arguments in \cite{TY3} simplify significantly; some extensions will be discussed in 
\cite{PT}. 
At a more general level, our work opens up the possibility of doing calculus on extremely rough surfaces by invoking the canonical smoothing that we are presenting here, and doing the analysis on the Ricci flow instead.

\emph{Acknowledgements:} PT was supported by EPSRC grant EP/T019824/1.
HY was supported by NSFC 11971451 and 2020YFA0713102. HY would like to thank Professor Yuxiang Li for a discussion on the result of Brezis-Merle \cite{BM}.
For the purpose of open access, the authors have applied a Creative Commons Attribution (CC BY) licence to any author accepted manuscript version arising.

\section{Construction of a maximally stretched solution from weak initial data}
\label{max_stretch_sect}

A central ingredient in the proof of Theorem \ref{main_exist_thm} is the following construction of a maximally stretched solution starting with 
a nonatomic Radon measure $\mu$ on the unit disc $D$ in the plane, or on the entire plane.

\begin{thm}[Existence of maximally stretched Ricci flow starting with $\mu$]
\label{DC_max_stretch_thm}
Let the Riemann surface $M$ be either the disc or the plane, and let $\mu$ be a 
nonatomic Radon measure on $M$. If $M$ is the disc we set $T=\infty$, while if $M$ is the plane then we set $T=\frac{1}{4\pi}\mu(M)\in [0,\infty]$.
Then there exists a smooth complete conformal Ricci flow $g_{max}(t)$ on $M$, for $t\in (0,T)$,
attaining $\mu$ as initial data in the sense that 
$$\mu_{g_{max}(t)}\weakto \mu\text{ as }t\downto 0$$
and so that if $\ti g(t)$, $t\in (0,\ti T)$, is any smooth conformal Ricci flow on $M$ 
(not necessarily complete)
that attains $\mu$ as initial data in the same sense, then 
$\ti T\leq T$ and 
$$g_{max}(t)\geq \ti g(t)\qquad\text{ for all }t\in (0,\ti T).$$
\end{thm}

In order to construct the maximally stretched solution of Theorem \ref{DC_max_stretch_thm} we need to analyse a potential corresponding to the flow that would coincide with the standard potential of K\"ahler Ricci flow if we were working on a closed surface. Understanding the right way of defining and localising the potential in the noncompact case is a key part of our task in this paper.




\begin{lem}[Construction of a potential]
\label{potential_construct_lem}
Suppose the Riemann surface $M$ is either the disc or the plane.
Suppose we have a smooth complete conformal Ricci flow $g(t)$ on $M$, $t\in (0,T)$, 
for $T\in (0,\infty]$, with conformal factor $u:M\times (0,T)\to(0,\infty)$,  taking a 
Radon measure $\mu$ on $M$ as initial data.
Then there exists a smooth potential function $\vph:M\times (0,T)\to \R$ satisfying
\beq
\label{vph_elliptic}
\lap \vph(t)=u(t)>0
\eeq
for each $t\in (0,T)$ and 
\beq
\label{potential_evol_eq_first}
\pl{\vph}{t}=\log\lap\vph
\eeq
throughout $M\times (0,T)$.
Moreover, there exists a function $\vph_0:M\to\R\union\{-\infty\}$ representing an element of $L^1_{loc}(M)$ such that 
$\vph(t)\to\vph_0$ both pointwise and in $L^1_{loc}(M)$ as $t\downto 0$.
Moreover, 
\begin{equation}
\label{vph0_elliptic}
\lap \vph_0=\mu
\end{equation}
weakly, and 
\beq
\label{vph_0_char_new}
\vph_0(x)=\lim_{r\downto 0} \fint_{B_r(x)} \varphi_0 dx,
\eeq
for all $x\in M$. 
%
\end{lem}


By virtue of \eqref{vph_elliptic} and \eqref{potential_evol_eq_first}, the potential $\vph$ and its initial values $\vph_0$ are only determined up to the addition of a time-independent harmonic function on $M$.

\cmt{Slightly subtle point that we are now omitting to explain explicitly: Although $\vph_0$ is defined pointwise in terms of the specific potential flow $\vph(t)$, if we only know it a.e. (in particular as an element of $L^2$ arising as a limit of a different potential flow) then \eqref{vph_0_char_new} allows us to reconstruct it at every point.}

\cmt{currently we don't say "upper semi-continuous" at this stage, to make things as streamlined as possible.}

\cmt{The intrinsic characterisation \eqref{vph_0_char_new}
of $\vph_0$ will make it easier to compare two solutions with the same initial $\vph_0$ in a pointwise sense.}

\cmt{We are purposefully not mentioning that the integrals in 
\eqref{vph_0_char_new} are monotonic in $r$, to avoid distraction.}

\cmt{note that the lemma below does not say anything about $\mu$, and in particular, does not say that $\lap\vph_0=\mu$}



To avoid repetition in Appendix \ref{sec:s2}, we extract a portion of the proof of Lemma \ref{potential_construct_lem} in the following.


\begin{lem}
\label{potential_sublem}
For $R>0$, $T\in (0,\infty]$, $\tau\in (0,T)$, let $u\in C^\infty(B_R\times (0,T))$ satisfy
    \[
    u(x,t)\geq \varepsilon t \qquad \text{on} \quad B_R\times (0,\tau]  
    \]
for some $\varepsilon>0$, and
    \begin{equation}
    \label{uL1_hyp}
    \int_0^\tau \int_{B_R} \abs{u} dx dt <C
    \end{equation}
for some $C>0$.
    If 
    \begin{equation}
    \label{vphbarvph_def}
    \varphi(x,t)= \varphi_\tau(x) + \int_\tau^t \log u(x,s) ds
    \end{equation}
    for some ${\varphi_\tau}\in C^\infty(\overline{B_R})$,
    then the pointwise limit $\varphi_0:=\lim_{t\downto 0} \varphi(t)$ taking values in  $\Real\cup \set{-\infty}$ exists and satisfies
\begin{enumerate}[(1)]
    \item\label{part1}
    $\varphi_0$ is upper semi-continuous on $B_R$.
    \item \label{part2}
    $\varphi_0$ is in $L^1(B_R)$ and $\varphi(t)\to \varphi_0$ in $L^1(B_R)$ as $t\downto 0$. 
    \item\label{part3}
    If we further assume that  $\Delta \varphi(t)\geq 0$
    for every $t\in (0,T)$, then 
    \begin{equation}
    \label{vph0_av_lim}
    \varphi_0(x) = \lim_{r\downto 0} \fint_{B_r(x)} \varphi_0 dx,   
    \end{equation}
    for all $x\in B_R$.
\end{enumerate}
\end{lem}

\begin{proof}[Proof of Lemma \ref{potential_sublem}]
To control $\vph$ in the limit $t\downto 0$, we split \eqref{vphbarvph_def} as 
$$\vph(x,t)=\vph_\tau(x)+\check\vph(x,t)+\hat\vph(x,t),$$
where 
$$\check\vph(x,t):=\int_\tau^t \chi_{\{u\leq 1\}}\log u(x,s)\,ds
\quad\text{ and }\quad
\hat\vph(x,t):=\int_\tau^t \chi_{\{u> 1\}}\log u(x,s)\,ds,
$$
are decreasing and increasing in $t$, respectively.
The function $\check\vph(t)$ converges uniformly on $B_R$ as $t\downto 0$ because we are assuming that $u\geq \ep t$
for $t\in (0,\tau]$,
so the function $\ti\vph:B_R\times (0,T)\to\R$ defined by
$$\ti\vph(t):= \check\vph(t)+ t(1-\log(\ep t))=
\tau(1-\log(\ep \tau))+
\int_\tau^t \left(\chi_{\{u\leq 1\}}\log u(x,s)-\log (\ep s)\right)ds
$$
\cmt{Note that below we are adopting the convention that $\vph_0\equiv -\infty$ would be a valid upper semi-continuous function. That is impossible here, but we say $\vph_0$ is upper semi-continuous before we establish that.}
is increasing for 
$t\leq \ep^{-1}\wedge\tau$ 
even though $\check \vph(t)$ is decreasing.
These considerations also tell us that
$$t\mapsto \vph(t)+t(1-\log(\ep t))=\vph_\tau(x)+\ti\vph(x,t)+\hat\vph(x,t)$$
is a pointwise increasing family of continuous functions for small $t$.
As any decreasing limit of continuous functions is upper semi-continuous, we obtain an upper semi-continuous function
$\varphi_0:=\lim_{t\downto 0} \varphi(t)$ as required for Part (\ref{part1}).
Meanwhile, 
$$\pl{\hat\vph}{t}=\chi_{\{u> 1\}}\log u(x,t)\leq u(x,t),$$
and so \eqref{uL1_hyp} tells us that $\hat\vph(t)$, and hence $\vph(t)$, converges in $L^1(B_R)$ as $t\downto 0$, as required for Part (\ref{part2}).

In the final part we are also told that $\vph(t)$ is subharmonic, and in particular 
$$\vph(x,t)\leq \fint_{B_r(x)} \vph(t) dx,$$
for all $x\in B_R$ and small enough $r>0$ so that $B_r(x)\subset\subset B_R$.
We can pass this inequality to the limit $t\downto 0$ to give
$$\vph_0(x)\leq 
\fint_{B_r(x)} \vph_0 dx,$$
and hence 
$$\vph_0(x)\leq \liminf_{r\downto 0}\fint_{B_r(x)} \vph_0 dx.$$
On the other hand, we have established that $\vph_0$ is upper semi-continuous, and so 
$$\vph_0(x)\geq \limsup_{r\downto 0}\fint_{B_r(x)} \vph_0 dx.$$
Combining these two inequalities gives Part (\ref{part3}).
\end{proof}

\begin{rmk}
Classical potential theory gives other ways of phrasing the proof, some of which give additional information that we do not require. For example, the integrals in \eqref{vph0_av_lim} are increasing in $r$.
See, for example, \cite[Section 3]{class_harm_anal}.
\end{rmk}

The following remark will be used in the proof of 
Lemma \ref{potential_construct_lem} and later also.
\begin{rmk}
\label{uovert_mono}
Whenever we have a smooth complete Ricci flow $g(t)$ on a surface $M$, $t\in (0,T)$, and we write it $g(t)=u(t)g_0$ for conformal factor $u$ and fixed background metric $g_0$, then for each $x\in M$, 
$$t\mapsto \frac{u(x,t)}{t}\quad\text{ is monotonically decreasing in }t.$$
The reason is that the Ricci flow equation \eqref{RF_eq}
implies that 
$$\pl{u}{t}=-2K_g u,$$
while because of the completeness, 
B.-L. Chen's scalar curvature estimate \cite{strong_uniqueness} tells us that the Gauss curvature satisfies the lower bound $K_{g(t)}\geq -\frac{1}{2t}$ for all $t\in (0,T)$.
Then 
$$\pl{}{t}\left(\frac{u}{t}\right)=-\frac{2u}{t}\left(K+\frac{1}{2t}\right)\leq 0.$$
\end{rmk}

\begin{proof}[Proof of Lemma \ref{potential_construct_lem}]
Define $\tau:=\frac{T}{2}\wedge 1$ and fix a smooth solution $\vph_\tau$ to the 
equation 
$$\lap\vph_\tau=u(\tau)$$
on $M$. Such a solution can be found using Lemma \ref{lem:poisson} and is  unique up to the addition of a harmonic function.
We define the required smooth potential $\vph:M\times (0,T)\to \R$ by
\beq
\label{ti_vph_def}
\vph(x,t)=\vph_\tau(x)+\int_\tau^t \log u(x,s)\,ds.
\eeq
By applying $\lap$ and using the Ricci flow equation \eqref{RF_eq_u} we obtain at each point $x\in M$ that 
\begin{equation}
\label{vph_t_eq}
\begin{aligned}
\lap \vph(t) &=\lap \vph_\tau+\int_\tau^t \lap\log u(s)\,ds\\
&=u(\tau)+\int_\tau^t \pl{u}{t}(s)\,ds\\
&=u(t),
\end{aligned}
\end{equation}
and so $\vph$ satisfies \eqref{vph_elliptic}.
Moreover, combining this with the time derivative of \eqref{ti_vph_def} tells
us that $\vph$ satisfies \eqref{potential_evol_eq_first}.

Pick $R>0$ such that $B_R:=B_R(0)\subset\subset M$.
Choose $\ep'>0$ so that $u(\tau)\geq \ep'$ on $B_R$.
By Remark \ref{uovert_mono}, for $\ep:=\ep'/\tau>0$, we then have 
$$u(t)\geq \ep t$$
throughout $B_R$, for all $t\in (0,\tau]$. We have established the first hypothesis of Lemma \ref{potential_sublem}.

Next, fix any $\ti R>R$ still with $B_{\ti R}\subset \subset M$.
Then for any $0<s\leq t\leq \tau$, we can appeal to 
\cite[Lemma 3.2]{TY3} to find that
\beq
\label{controlled_mass_gain}
\Vol_{g(t)}(B_R)\leq (t-s)\eta +\Vol_{g(s)}(B_{\ti R})
\eeq
for some $\eta>0$ depending only on $\ti R/R$.
In particular, by taking a limsup as $s\downto 0$ we obtain
\beq
\label{uL1}
\int_{B_R}u(t)dx=
\Vol_{g(t)}(B_R)\leq \eta t+\mu(\overline{B_{\ti R}})\leq L,
\eeq
where $L<\infty$ is independent of $t$ (keeping in mind that $t\leq \tau\leq 1$).
This implies hypothesis \eqref{uL1_hyp} of Lemma \ref{potential_sublem}.

Applying Lemma \ref{potential_sublem} then gives $\vph_0$ and all its desired properties except for
the equation \eqref{vph0_elliptic}.
By \eqref{vph_elliptic}, for every $\xi\in C_c^2(M)$ we have
$$\int_M \vph(t)\lap \xi dx =\int_M u(t)\xi dx,$$
and by taking a limit $t\downto 0$ we obtain
$$\int_M \vph_0\lap \xi dx=\int_M \xi d\mu,$$
i.e. $\vph_0$ is a weak solution of
$$\lap\vph_0=\mu$$
as required.
\end{proof}

In practice we must contemplate other solutions $\psi$ to the potential function evolution equation \eqref{potential_evol_eq_first} with the same initial data $\vph_0$ that was induced in Lemma \ref{potential_construct_lem}.
It will be important to verify that such $\psi$ cannot jump up unreasonably from $\vph_0$ as $t$ lifts off from zero, and that is the content of the following lemma.

\cmt{New version gets rid of $\mu$ and changes from mollification.
We got rid of the `decreasing limit' statement. We might want to reword to allow a disc of any radius, depending on application in appendix (earlier applications are fine like this).}



\begin{lem}[Space-time upper semi-continuity]
\label{upper_semi_lem}
Suppose the Riemann surface $M$ is either the disc or the plane.
Suppose, for $T>0$, that we have a smooth function $\psi:M\times (0,T)\to\R$ such that 
$$\lap\psi(t)\geq 0$$
throughout, and so that 
\beq
\label{another_L1_hyp}
\psi(t)\to \vph_0\quad\text{ in }L^1_{loc}(M)
\eeq
as $t\downto 0$, where $\vph_0\in L^1_{loc}(M)$ 
admits a representative for which
$$\vph_0(x)=\lim_{r\downto 0} \fint_{B_r(x)} \varphi_0 dx$$
for all $x\in M$.
Then for all $x\in M$,  all $x_n\to x$ and all $t_n>0$ with $t_n\downto 0$, we have 
$$\vph_0(x)\geq \limsup_{n\to\infty}\psi(x_n,t_n).$$
\end{lem}

\begin{proof}
By subharmonicity of $\psi(t_n)$, the mean value inequality gives
$$\psi(x_n,t_n)\leq \fint_{B_r(x_n)}\psi(t_n)dx$$
for sufficiently small $r>0$ and all $n$.
Therefore
$$\limsup_{n\to\infty}\psi(x_n,t_n)\leq \lim_{n\to\infty}\fint_{B_r(x_n)}\psi(t_n)dx
=\fint_{B_r(x)}\vph_0 dx$$
where the final equality here follows from the $L^1_{loc}(M)$ convergence hypothesis
\eqref{another_L1_hyp}.
We can then take a limit $r\downto 0$ to conclude.
\end{proof}



We will prove Theorem \ref{DC_max_stretch_thm} using the following lemma. Note that in this lemma, the flow $g_{max}(t)$ depends \emph{a priori} on $g(t)$, although ultimately this dependency will be lifted.


\begin{lem}[Maximal potential function construction]
\label{core_lem}
Let the Riemann surface $M$ be either the disc or the plane, and let $\mu$ be a 
nonatomic Radon measure on $M$.
If $M$ is the disc we set $T=\infty$, while if $M$ is the plane then we set $T=\frac{1}{4\pi}\mu(M)\in [0,\infty]$.
Suppose $g(t)$ is a smooth  conformal Ricci flow  on $M$, for $t\in (0,\ti T)$,
$\ti T>0$,
that attains $\mu$ as initial data in the sense that $\mu_{g(t)}\weakto \mu$
as $t\downto 0$.
Then $\ti T\leq T$ and there exists a smooth complete conformal Ricci flow $g_{max}(t)$ on $M$, for $t\in (0,T)$ that also attains $\mu$ as initial data, and which satisfies
$$g_{max}(t)\geq g(t)\quad\text{ for all }t\in (0,\ti T).$$ 
Let $\vph\in C^\infty(M\times (0,T))$ be a potential coming from Lemma \ref{potential_construct_lem}, together with the corresponding $\vph_0$, corresponding to $g_{max}(t)$. 
Then if $\psi\in C^\infty(M\times (0,T))$ is any other potential, i.e. 
\beq
\label{potential_evol_eq}
\lap\psi(t)\geq 0\qquad\text{ and }\qquad 
\pl{\psi}{t}=\log\lap\psi,
\eeq
also with $\psi(t)\to\vph_0$ in $L^1_{loc}(M)$ at $t\downto 0$, 
then 
$$\psi(t)\leq \vph(t)$$
throughout $M\times (0,T)$.
\end{lem}


\cmt{In our last paper, \cite{TY3}, we should add explicitly in Lemma 4.2 (and Lemma 4.3 automatically) that $\mu$ is nontrivial. Be careful if we use the $\mu$ trivial case of those lemmas (e.g. on a disc)}


\begin{proof}[{Proof of Lemma \ref{core_lem}}]
The only case in which $T=0$ is the case that $M$ is the plane and $\mu$ is the trivial measure.
Our sole task in this case is to show the nonexistence of any 
smooth  conformal Ricci flow  $g(t)$ on $M$, for $t\in (0,\ti T)$,
$\ti T>0$,
that attains the trivial measure $\mu$ as initial data.
Suppose such a flow exists. Then
for each $\ti R>R>0$, and 
$0<s\leq t< \ti T$,  \cite[Lemma 3.2]{TY3} gives us
\beq
\label{controlled_mass_gain_again}
\Vol_{g(t)}(B_R)\leq (t-s)\eta +\Vol_{g(s)}(B_{\ti R})
\eeq
where the proof establishes that $\eta$ is the volume of $B_R$ with respect to the complete conformal hyperbolic metric on $B_{\ti R}$. 
Sending $s\downto 0$ gives
$$\Vol_{g(t)}(B_R)\leq \eta t$$
but $\eta\to 0$ as  $\ti R\to\infty$, so
$\Vol_{g(t)}(B_R)=0$. Sending $R\to\infty$ then 
gives $\Vol_{g(t)}(\R^2)=0$, which is a contradiction.
Thus from now on we may assume that $T>0$,
though $\mu$ could still be the trivial measure on the disc.

As a warm-up to the techniques involved in constructing $g_{max}(t)$ we first reduce the lemma to the case that the given Ricci flow $g(t)$ is complete, and $\ti T=T$. 
This will be possible by replacing $g(t)$ by the flow $\hat g(t)$ arising in the following claim.

{\bf Claim 1:}
There exists a smooth conformal \emph{complete} Ricci flow $\hat g(t)$ on $M$, for $t\in (0,T)$,
weakly attaining $\mu$ as initial data, and 
with $\hat g(t)\geq g(t)$ for $t\in (0,T\wedge\ti T)$. Moreover, if $T<\infty$
then $\mu_{\hat g(t)}(M)\to 0$ as $t\upto T$ so we must have $\ti T\leq T$.

\emph{Proof of Claim 1:}
For each $\de\in (0,\ti T)$, we take $g(\de)$ as smooth (possibly incomplete) initial data for a Ricci flow 
on $M$. The previous two-dimensional theory \cite{GT2, ICRF_UNIQ} tells us that 
there exists a unique  Ricci flow $g_\de(t)$, $t\in [\de,T_\de)$ 
(for some $T_\de\in (\de,\infty]$)
that is complete for $t\in (\de,T_\de)$,
and for which $g_\de(\de)=g(\de)$.
Moreover, the earlier theory tells us that if $T_\de<\infty$ then 
$\mu_{g_\de(t)}(M)\to 0$ as $t\upto T_\de$, and that
$g_\de(t)\geq g(t)$ at times at which both flows exist because $g_\de(t)$ is maximally stretched \cite{GT2}, i.e., it lies above any other smooth Ricci flow with the same, or lower, initial data. 
We conclude that $g_\de(t)$ must exist
at least for $t\in [\de,\ti T)$. 


Moreover, for $0<\de_1\leq\de_2<\ti T$, 
appealing again to the maximally stretched property of $g_{\de_1}(t)$,
we have 
$g_{\de_1}(t) \geq g_{\de_2}(t)$ for $t\in [\de_2,\ti T)$, i.e.
for times $t$ at which both flows exist.

\cmt{Below, when we apply \cite[Lemma 4.2]{TY3}, we may want to rephrase in the case that $\mu$ is the trivial measure on the disc. This depends on how we eventually correct the statement of that lemma.}

If we pick a sequence $\de_i\downto 0$, then
$t\mapsto g_{\de_i}(t+\de_i)$
gives a sequence of Ricci flows whose initial data converges weakly to $\mu$.
We can then apply \cite[Lemma 4.2]{TY3} to deduce that after passing to a subsequence
in $i$, there exists $T_0\in (0,\ti T)$ such that the flows converge smoothly locally on 
$M\times (0,T_0)$ 
to a new Ricci flow $\hat g(t)$ on $M$, for $t\in (0,T_0)$, that  has $\mu$ as its initial data (as for $g(t)$) but that is now  complete. 
Translating the $i$th flow by $\de_i$ in time, we find that the Ricci flows $g_{\de_i}(t)$
also converge smoothly locally on $M\times (0,T_0)$ 
to $\hat g(t)$. Because $g_{\de_i}(t)\geq g(t)$, we also have 
$\hat g(t)\geq g(t)$ on $M\times (0,T_0)$. We would now like to extend $\hat g(t)$ in time.
A basic consequence of $\hat g(t)$ attaining $\mu$ as initial data is that
$$\mu(M)\leq\liminf_{t\downto 0}\mu_{\hat g(t)}(M).$$
We can then extend $\hat g(t)$ by taking the unique complete Ricci flow starting with 
$\hat g(\eta)$, for small $\eta\in (0,T_0)$, using \cite{GT2, ICRF_UNIQ}. This extension exists for all time if $M$ is the disc, while if $M$ is the plane then it exists for a
time 
$$\frac{\mu_{\hat g(\eta)}(M)}{4\pi}
\xrightarrow{\liminf}
\frac{\mu(M)}{4\pi}=T$$
as $\eta\downto 0$.
Indeed, in this case the extension will satisfy
$$\mu_{\hat g(t)}(M)=\left(1-\frac{t}{T}\right)\mu(M)\to 0$$
as $t\upto T$.
Because $\hat g(t)$ is maximally stretched, we can extend the inequality
$\hat g(t)\geq g(t)$ to $M\times (0,T\wedge\ti T)$.
Then $g(t)$ cannot exist beyond time $T$ since it would have no volume left, 
and so $\ti T\leq T$. 
\hfill \textit{End of proof of Claim 1.}

\medskip


As a result,  we may 
replace $g(t)$, for $t\in (0,\ti T)$, by the complete flow $\hat g(t)$ for $t\in (0,T)$,
as claimed before Claim 1. We write $u(t)$ for the conformal factor of what we are now calling $g(t)$.



We now turn to constructing the $g_{max}(t)$ required in the lemma as the limit of flows defined on smaller domains. At this stage of the argument we are unsure whether or not the Ricci flow we have just constructed in Claim 1 (now renamed $g(t)$) can serve as the required $g_{max}(t)$. 
It will turn out to be the same, but this fact is not clear at this stage.

In the following, 
we set $R_*=1$ if $M$ is the disc, and $R_*=\infty$ if $M$ is the plane, so that 
$B_R\subset\subset M$ for each $R\in (0,R_*)$.

{\bf Claim 2:}
For $R\in (0,R_*)$, there exists a smooth conformal complete Ricci flow $g_R(t)$ on $B_R$, $t\in (0,\infty)$, with conformal factor $u_R(t)$, such that 
$\mu_{g_R(t)}\weakto \mu$ weakly as $t\downto 0$, where $\mu$ is restricted to $B_R$, 
and with the monotonicity property that if $0<R_1<R_2<R_*$ then 
$$u_{R_1}(t)\geq u_{R_2}(t)\quad \text{ throughout }B_{R_1} 
\text{ for }t\in (0,\infty),$$ 
and so that 
$$u_{R_2}(t)\geq u(t) \quad \text{ throughout }B_{R_2}\text{ for }t\in (0,T).$$

\emph{Proof of Claim 2:}
For each $R\in (0,R_*)$, and each $\de\in (0,T)$, take the unique 
instantaneously complete Ricci flow on $B_R$, for $t\in [\de,\infty)$, 
with conformal factor denoted by $u_{R,\de}$ satisfying $u_{R,\de}(\de)=u(\de)$
on $B_R$, where we recall that $u(t)$ is the conformal factor of $g(t)$.

Certainly $u_{R,\de}(t)\geq u(t)$ for all $t\in [\de,T)$ because 
an instantaneously complete solution (with smooth initial metric) 
such as that corresponding to $u_{R,\de}(t)$
is always larger than any other solution with the same initial data as a result of the maximally stretched property \cite{GT2}.
Similarly, 
for $0<\de_1\leq\de_2< T$, 
because $u_{R,\de_1}(\de_2)\geq u(\de_2)=u_{R,\de_2}(\de_2)$
on $B_R$, 
we have 
\beq
\label{delta_mono}
u_{R,\de_1}(t)\geq u_{R,\de_2}(t)\quad\text{ on }B_R\text{ for all }t\geq\de_2.
\eeq
Moreover,
if $0<R_1<R_2<R_*$ then 
$u_{R_1,\de}(t)\geq u_{R_2,\de}(t)$, throughout $B_{R_1}$, for all 
$t\in [\de,\infty)$ because 
$u_{R_1,\de}(t)$ is maximally stretched on $B_{R_1}$.

\cmt{again, we may want to eventually rephrase the wording below depending on the way that we end up handling the $\mu\equiv 0$ case of \cite[Lemma 4.2]{TY3}}

We can now mimic the proof of Claim 1 and 
appeal to \cite[Lemma 4.2]{TY3}, with $M$ there equal to $B_R$ here, 
to find that there exists $\de_i\downto 0$ such that 
the conformal factors $u_{R,\de_i}(t)$ converge smoothly locally on $B_R\times (0,T_R)$ 
(some $T_R>0$)
to the conformal factor $u_R:B_R\times (0,T_R)\to (0,\infty)$ of a smooth conformal complete Ricci flow $g_R(t)$ on $B_R$ whose volume measures converge weakly to (the restriction of)
$\mu$ as $t\downto 0$.
By construction, the monotonicity \eqref{delta_mono} of $u_{R,\de}(t)$ in $\de$ implies that 
$u_{R,\de}(t)\leq u_R(t)$ for $t\in [\de,T_R)$.


Because of the phrasing of \cite[Lemma 4.2]{TY3},
this only gives the compactness on some time interval $(0,T_R)$. However, 
if we can verify that we have  local positive lower and upper bounds for 
$u_{R,\de}(t)$ on $B_R\times (T_R/2,\infty)$ that are uniform as $\de\downto 0$
then we can invoke parabolic regularity theory and pass to a subsequence to 
get smooth local convergence of $u_{R,\de_i}(t)$ to an extended 
$u_R:B_R\times (0,\infty)\to (0,\infty)$.
Indeed, by invoking Yau's Schwarz lemma  and 
B.-L. Chen's scalar curvature estimate   $K_{g(t)}\geq -\frac{1}{2t}$ from \cite{strong_uniqueness}, as originating in \cite{GT1},
we have uniform lower bounds 
$u_{R,\de}(t)\geq 2(t-\de)h_R$, where $h_R$ is the conformal factor 
of the unique complete conformal hyperbolic metric on $B_R$, that are valid for all $t\geq \de$. Moreover, for any $r\in (0,R)$ if we pick sufficiently large $C_r<\infty$
so that $u_R(T_R/2)< C_r h_r$ on $B_r$, then also for $\de\in (0,T_R/2)$ we have
$u_{R,\de}(T_R/2)\leq u_R(T_R/2)< C_r h_r$ on $B_r$, and these upper bounds propagate forwards to give
$$u_{R,\de}(t+T_R/2)\leq  (C_r+2t) h_r$$ 
on $B_r$, by the maximally stretched property of the Ricci flow 
$t\mapsto (C_r+2t) h_r$.
These lower and upper bounds on the conformal factors $u_{R,\de}(t)$
suffice to give compactness on $B_R\times (0,\infty)$.

Moreover, these limit flows $u_R(t)$ inherit the ordering that 
$u_{R}(t)\geq u(t)$ on $B_R$ for all $t\in (0,T)$,
and also that 
if $0<R_1<R_2<R_*$ then 
$u_{R_1}(t)\geq u_{R_2}(t)$, throughout $B_{R_1}$, for all 
$t\in (0,\infty)$.
\hfill \textit{End of proof of Claim 2.}

\medskip

If we restrict our attention to an arbitrary compact set $K$ in $M$, the inequalities
for the conformal factors from Claim 2 are enough to obtain  positive lower and upper bounds on the conformal factors $u_R$ on $K$  that are 
uniform in $R$ as $R\upto R_*$. 
Thus we can find a sequence $R_i\upto R_*$ so that 
$u_{R_i}(t)$ converges smoothly locally on $M\times (0,T)$ to a limit $u_{max}(t)$, 
which is the conformal factor of a Ricci flow $g_{max}(t)$.
By construction, we have $u(t)\leq u_{max}(t)$ on $M$ for all $t\in (0,T)$.
One consequence is that $g_{max}(t)$  inherits the completeness from $g(t)$.
We also have, for all $R\in (0,R_*)$, that $u_{max}(t)\leq u_{R}(t)$ throughout
$B_R\times (0,T)$.

Furthermore, because $u_{max}(t)$ is sandwiched between $u(t)$ and $u_R(t)$, both of which weakly attain $\mu$ as initial data, we see that $u_{max}(t)$ does also.

We have constructed the Ricci flow $g_{max}(t)$ that is required in the lemma, 
and this induces the potential $\vph$ (with initial data $\vph_0$) by Lemma \ref{potential_construct_lem}. It remains to investigate how $\vph$ relates to the potential $\psi$. 
In order to do this, we will relate $\vph$ to approximating potentials on subdomains of $M$. For each $R\in (0,R_*)$, we feed the Ricci flows with conformal factors $u_R$ into Lemma \ref{potential_construct_lem}, with $M$ there equal to the ball $B_R$ here, and $T$ there equal to $\infty$ here, to yield a potential 
$\vph_R\in C^\infty(B_R\times (0,\infty))$ satisfying 
$$\lap \vph_R(t)=u_R(t)>0\qquad\text{ and }\qquad \pl{\vph_R}{t}=\log\lap\vph_R$$
throughout, and with initial data $(\vph_R)_0$ satisfying $\lap(\vph_R)_0=\mu$
weakly. Since $\vph_0-(\vph_R)_0$ is weakly harmonic on $B_R$, by redefining $\vph_R$ by adding a time-independent harmonic function we may assume that 
$(\vph_R)_0$ agrees with $\vph_0$. We emphasise that by \eqref{vph_0_char_new}, not only do
we have $\vph_R(t)\to\vph_0$ in $L^1_{loc}(B_R)$ as $t\downto 0$, but also 
the convergence holds pointwise.

\cmt{slightly subtle above: $\vph_0-(\vph_R)_0$ is being called *weakly* harmonic to emphasise a little that it agrees with a harmonic function a.e.}

The potentials $\vph_R$ inherit the monotonicity of $u_R$ with respect to $R$. Indeed, for $0<R_1<R_2<R_*$ and $x\in B_{R_1}$ with $\vph_0(x)>-\infty$, we have
$$\vph_{R_1}(x,t)=\vph_0(x)+\int_0^t \log u_{R_1}(s)ds
\geq \vph_0(x)+\int_0^t \log u_{R_2}(s)ds=\vph_{R_2}(x,t),$$
for every $t>0$,
and then by continuity, for  every point $x\in B_{R_1}$ and every $t>0$ we have
$$\vph_{R_1}(x,t)\geq \vph_{R_2}(x,t).$$
Similarly, for $R\in (0,R_*)$, because $u_{max}(t)\leq u_{R}(t)$ throughout
$B_R\times (0,T)$, we have
\beq
\label{vphRvph}
\vph_{R}\geq \vph
\eeq
throughout $B_R\times (0,T)$.
%
%
This inequality will have multiple applications. The first is that we can combine with the monotonicity of $\vph_R$ with respect to $R$ in order to define 
$\ti\vph: M\times(0,T)\to\R$ by
$$\ti\vph(x,t)=\lim_{R\upto R_*} \vph_R(x,t),$$
and to deduce that 
\beq
\label{ti_vph_vph_ineq}
\ti\vph\geq \vph.
\eeq
Moreover, the function $\ti\vph$ will be smooth. To see this, pick an arbitrary compact set $K\subset M\times(0,T)$ and choose $\ti R\in (0,R_*)$ sufficiently large so that 
$K\subset B_{\ti R}\times (0,T)$.
For $R\in [\ti R,R_*)$, the potentials $\vph_R$, restricted to $B_{\ti R}\times (0,T)$, will be sandwiched between $\vph$ and $\vph_{\ti R}$. 
Meanwhile, we already showed that $u_{R_i}(t)$ converges smoothly locally on $M\times (0,T)$. Elliptic regularity theory applied  to the equation
$$\lap\vph_{R_i}(t)=u_{R_i}(t),$$
for sufficiently large $i$, gives  control on all spatial derivatives of 
$\vph_{R_i}(t)$ over $K$ that is uniform as $i\to\infty$. By appealing to the equation 
$$\pl{\vph_{R_i}}{t}=\log\lap\vph_{R_i}$$
we then obtain  $C^k$ space-time control of $\vph_{R_i}$ over $K$ that is uniform in $i$.
By passing to a further subsequence we see that $\vph_{R_i}$ converges smoothly locally, and thus its pointwise limit $\ti \vph$ is smooth.


We claim now that we have equality in the inequality \eqref{ti_vph_vph_ineq}.
Because we have established that $\ti\vph$ is smooth, it suffices to prove the equality at points $x\in M$ where $\vph_0>-\infty$ and at times $t\in (0,T)$.
Choose $\ti R>|x|$ so that $x\in B_{\ti R}$.
For arbitrary $\ep>0$, we can pick $\de\in (0,t)$ sufficiently small so that 
$$|\vph(x,\de)-\vph_0(x)|<\ep/3$$
and 
$$|\vph_{\ti R}(x,\de)-\vph_0(x)|<\ep/3.$$
For $R\in (\ti R,R_*)$, $\vph_R(x,\de)$ is sandwiched above $\vph(x,\de)$
and below $\vph_{\ti R}(x,\de)$, and so 
$$|\vph_{R}(x,\de)-\vph_0(x)|<\ep/3$$
also.
Meanwhile, 
$$\vph_{R_i}(x,t)-\vph_{R_i}(x,\de)=\int_\de^t \log u_{R_i}(x,s)ds$$
and 
$$\vph(x,t)-\vph(x,\de)=\int_\de^t \log u_{max}(x,s)ds$$
and by the smooth local convergence of $u_{R_i}$ to $u_{max}$, for sufficiently large $i$ 
we have
$$|(\vph_{R_i}(x,t)-\vph_{R_i}(x,\de))-(\vph(x,t)-\vph(x,\de))|<\ep/3.$$
In particular, for sufficiently large $i$ also to ensure that $R_i\geq \ti R$, we have
\beqa
|\vph_{R_i}(x,t)-\vph(x,t)|
&\leq 
|(\vph_{R_i}(x,t)-\vph_{R_i}(x,\de))-(\vph(x,t)-\vph(x,\de))|\\
&\quad +|\vph_{R_i}(x,\de)-\vph_0(x)|
 +|\vph(x,\de)-\vph_0(x)|\\
&< \ep.
\eeqa
Because $\ep>0$  was arbitrary,
taking the limit $i\to\infty$ forces
$$\vph(x,t)=\lim_{i\to\infty}\vph_{R_i}(x,t)=:\ti\vph(x,t)$$
as required.

The next step is to show that as $R$ increases, making $\vph_R(x,t)$ decrease, 
these potentials will remain above the  globally-defined potential $\psi$ considered 
in the statement of the lemma, once restricted to $B_R$.

{\bf Claim 3:}
For every $R\in (0,R_*)$, 
we have 
$$\psi\leq \vph_R$$
throughout $B_R\times (0,T)$.

Notice that if this claim is true then we can apply it for $R=R_i$ and let $i\to\infty$ to obtain that for every $x\in M$ and $t\in (0,T)$ we have
$$\psi(x,t)\leq \lim_{i\to\infty}\vph_{R_i}(x,t)=\vph(x,t)$$
which will complete the proof of Lemma \ref{core_lem}.

\emph{Proof of Claim 3:}
It suffices to prove the claim over $B_R\times (0,T_1)$ for arbitrary $T_1\in (0,T)$; we fix such a $T_1$.

Since $u_R$ is the conformal factor of a complete Ricci flow on $B_R$, $t>0$, 
we can appeal again to the idea from \cite{GT1} and use Yau's Schwarz lemma to deduce that
$$u_R(t)\geq 2th_R$$
on $B_R$, for every $t>0$, where $h_R$ is the conformal factor of the unique complete conformal hyperbolic metric on $B_R$. 
This lower bound for $u_R$ gives us lower control on $\vph_R$ because
for every $x\in B_R$ 
and every $t\geq \tau>0$, we have
\beqa
\label{vph_R_lower}
\vph_{R}(x,t) &= \vph_R(x,\tau)+\int_\tau^t \log u_{R}(x,s)ds\\
&\geq \vph_R(x,\tau)+\int_\tau^t \log (2sh_{R}(x))ds\\
&= \vph_R(x,\tau)+(t-\tau)\log (2h_R(x)) +t(\log t -1)-\tau(\log\tau -1).
\eeqa
Explicitly we have
\begin{equation}
\label{eqn:hR}
h_R(x)=\left(\frac{2R}{R^2-|x|^2}\right)^2\geq \frac{4}{R^2},
\end{equation}
and so \eqref{vph_R_lower} gives both control on how $\vph_R(t)$ blows up near the boundary of $B_R$ and also on how fast $\vph_R(t)$ can decrease in time.
For the latter, for points $x\in B_R$ with $\vph_0(x)>-\infty$, 
we can send $\tau\downto 0$ in \eqref{vph_R_lower} to deduce that 

%
\beqa
\label{moon}
\vph_{R}(x,t) &\geq \vph_0(x)+ t(\log (2h_R(x))+\log t -1\big)\\
&\geq \vph_0(x)- F(t),
\eeqa
where 
$$\textstyle
F(t):=t(1-\log \frac{8t}{R^2}).$$
As a consequence, for each $\de>0$ we can define a modified potential
$$\vph_{R,\de}(t):=\vph_R(t+\de)+F(\de)+\de(1+t)$$
on $B_R\times (0,\infty)$ so that $\vph_{R,\de}(0)$ is a smooth function on $B_R$ that is bounded below, and so that 
\beq
\label{vphR_initial_control}
\vph_{R,\de}(0)\geq \vph_0+\de.
\eeq

\cmt{although we earlier assumed that $\vph_0(x)>-\infty$, this last statement is true generally}


{\bf Subclaim:} For  $\de>0$ chosen sufficiently small so that $T_1+\de<T$, we have 
\beq
\label{subclaim_est}
\psi< \vph_{R,\de}
\eeq
throughout $B_R\times (0,T_1)$.

Observe that if we can prove this subclaim, then by taking a limit $\de\downto 0$
we will have proved Claim 3, and hence the whole lemma.

At a heuristic level, \eqref{subclaim_est} is true at $t=0$ by \eqref{vphR_initial_control} because $\psi$ takes $\vph_0$ as initial data. However, $\psi$ only admits initial data as a $L^1_{loc}$ limit, so this statement requires care.

We will establish the subclaim using a classical maximum principle approach. In lieu of an ordering like \eqref{subclaim_est}
on the boundary $\partial B_R$, we will need to show that $\vph_{R,\de}-\psi$ blows up in an appropriately uniform sense near the boundary and we will achieve this by proving an estimate of the form 
\beq
\label{boundary_control}
\vph_{R,\de}(t)-\psi(t)\geq \frac{\de}{2}\log(h_R)-C_1,
\eeq
throughout $B_R\times (0,T_1]$, where 
$C_1<\infty$ is independent of the point in $B_R\times (0,T_1]$ considered.

There are three ingredients for this. 
The first is that 
the potential $\psi$ has a finite upper bound $\overline{L}$ over $\overline{B_R}\times (0,T_1]$.
If this were not the case then we would be able to find a sequence 
$x_n\in \overline{B_R}$, and times $t_n\downto 0$, with
$\psi(x_n,t_n)\to\infty$. By passing to a subsequence, we could assume that $x_n\to x$ for some $x\in \overline{B_R}$.
But 
Lemma \ref{upper_semi_lem} would tell us that 
$$\limsup_{n\to\infty}\psi(x_n,t_n)\leq \vph_0(x)
<\infty,$$
giving a contradiction. We conclude that 
$$\psi\leq \overline{L}\qquad\text{ on }\overline{B_R}\times (0,T_1].$$
The second ingredient for \eqref{boundary_control} is that $\vph$ has a finite \emph{lower} bound $\underline{L}$ over
$\overline{B_R}\times [\de/2,T_1+\de]$. 
This is simply by continuity of $\vph$ and compactness of the space-time region considered.
An immediate consequence (see \eqref{vphRvph}) is that the larger function $\vph_R$ has the same lower bound over the part of the region considered where it is defined, i.e.
$$\vph_R\geq \underline{L}\qquad \text{ on }B_R\times [\de/2,T_1+\de].$$
The third and final ingredient is 
\eqref{vph_R_lower} with $\tau=\de/2$. This implies that
throughout $B_R$ and 
for $t\in [\de/2,T_1+\de]$ we have
\beqa
\vph_{R}(t) &\geq \textstyle \vph_R(\frac{\de}{2})+(t-\frac{\de}{2})\log (2h_R) +t(\log t -1)
-\frac{\de}{2}(\log\frac{\de}{2} -1)\\
&\geq \textstyle \underline{L} + (t-\frac{\de}{2})\log\left(\frac{R^2h_R }{4}\right)
+(t-\frac{\de}{2})\log \frac{8}{R^2} +t(\log t -1)
-\frac{\de}{2}(\log\frac{\de}{2} -1)\\
&\geq \textstyle(t-\frac{\de}{2}) \log\left(\frac{R^2h_R }{4}\right)-C_0,
\eeqa
where $C_0<\infty$ depends on $\underline{L}$, $T_1$, $R$ and  $\de$.
In particular, for $t\in (0,T_1]$ this gives
\beqa
\vph_{R,\de}(t)-\psi(t)
&\geq \frac{\de}{2}\log\left(\frac{R^2h_R }{4}\right)-C_0+F(\de)-\overline{L}\\
&\geq \frac{\de}{2}\log(h_R)-C_1,
\eeqa
where $C_1<\infty$ depends on $\overline{L}$, $\underline{L}$, $T_1$, $R$ and  $\de$.
This completes the proof of \eqref{boundary_control}.

We now complete the proof of the subclaim (and hence the whole lemma) by contradiction.
If it fails for some valid $\de$ then we can find 
a point in $B_R\times (0,T)$ where $\psi\geq\vph_{R,\de}$. We claim that then we can find a point
$(x_0,t_0)\in B_R\times (0,T)$ so that $\psi(x_0,t_0)=\vph_{R,\de}(x_0,t_0)$ but so that for all $t\in (0,t_0)$ we have $\psi(t)<\vph_{R,\de}(t)$ throughout $B_R$.
To find this point $(x_0,t_0)$, first define 
$$t_0:=\inf\{t\in (0,T)\ :\ \exists\,x\in B_R\text{ with }
\psi(x,t)\geq \vph_{R,\de}(x,t)\}\in [0,T).$$
We can then take sequences $t_n\downto t_0$ and $x_n\in B_R$ with 
$$\psi(x_n,t_n)\geq \vph_{R,\de}(x_n,t_n).$$
After passing to a subsequence we may assume that $x_n\to x_0\in B_R$.
Note that $x_0$ cannot lie on $\partial B_R$ because of the 
control near the boundary given by \eqref{boundary_control}.

If $t_0=0$, then we appeal to Lemma \ref{upper_semi_lem}. 
This gives us the final inequality of
$$\vph_{R,\de}(x_0,0)=\lim_{n\to\infty} \vph_{R,\de}(x_n,t_n)
\leq \limsup_{n\to\infty} \psi(x_n,t_n)
\leq \vph_0(x_0),$$
which contradicts \eqref{vphR_initial_control}.
Thus $t_0>0$ and we find that 
\beq
\label{soon_to_be_equality}
\vph_{R,\de}(x_0,t_0)=\lim_{n\to\infty} \vph_{R,\de}(x_n,t_n)
\leq \lim_{n\to\infty} \psi(x_n,t_n) = \psi(x_0,t_0).
\eeq
But by definition of $t_0$, for every $t\in (0,t_0)$ we have $\psi(t)<\vph_{R,\de}(t)$, 
so 
\beq
\label{second_deriv_prep}
\psi(t_0)\leq \vph_{R,\de}(t_0)
\eeq
and we must have equality throughout \eqref{soon_to_be_equality}, and we have found the required point $(x_0,t_0)$.

It remains to compute derivatives at this point $(x_0,t_0)$ in order to generate a contradiction.
Applying the second derivative test to $\vph_{R,\de}(t_0)-\psi(t_0)$ at $x_0$, where it achieves its minimum value of zero, we find that 
\beq
\label{lap_ineq}
\lap(\vph_{R,\de}-\psi)\geq 0
\eeq
at $(x_0,t_0)$.
By the characterisation of $t_0$, we must have
\beq
\label{ddt_ineq}
\pl{}{t}\left(\vph_{R,\de}-\psi\right)\leq 0
\eeq
at $(x_0,t_0)$.
But 
$$\pl{}{t}\vph_{R,\de}(x,t)=\pl{}{t}\vph_{R}(x,t+\de)+\de
=\log\lap\vph_{R}(x,t+\de)+\de
=\log\lap\vph_{R,\de}(x,t)+\de
$$
while 
$$\pl{\psi}{t}=\log\lap\psi.$$
Evaluating at $(x_0,t_0)$, subtracting and applying \eqref{lap_ineq} gives
$$\pl{}{t}\left(\vph_{R,\de}-\psi\right)
=\log\lap\vph_{R,\de}+\de-\log\lap\psi
\geq \de,
$$
which contradicts \eqref{ddt_ineq}.
\end{proof}

We can derive Theorem \ref{DC_max_stretch_thm} rapidly from Lemma \ref{core_lem}.


\begin{proof}[{Proof of Theorem \ref{DC_max_stretch_thm}}]
As a preliminary step, we consider the case that $\mu$ is the trivial measure 
and $M$ is the plane. In this case $T=0$ so the existence of $g_{max}(t)$ is a vacuous statement. 
Moreover, it is a consequence of Lemma \ref{core_lem} that no Ricci flow $\ti g(t)$ as in the theorem can exist because $\ti T>0$ is impossible.
This completes the proof in that case.


In all remaining cases, we have $T>0$. We start by picking a smooth complete conformal Ricci flow $g(t)$ on $M$ for $t\in (0,T)$ with $\mu$ as initial data. 
In the case that $M$ is the disc and $\mu$ is the trivial measure
we simply take $g(t)$ to be the big-bang Ricci flow with conformal factor $2th$, where $h=h_1$ (as in \eqref{eqn:hR}) is the conformal factor of the Poincar\'e metric. 


In the remaining cases in which $\mu$ is not the trivial measure,
we can use  our previous existence result \cite[Theorem 1.2]{TY3}
to find a smooth complete conformal Ricci flow $g(t)$ for $t\in (0,T)$ such that 
$\mu_{g(t)}\weakto \mu$ as $t\downto 0$.
This means that we can apply Lemma \ref{core_lem}, even for $\ti T=T$. 
The output of the lemma is a new Ricci flow $g_{max}(t)\geq g(t)$
for $t\in (0,T)$ with $\mu$ as initial data, and also a potential $\vph$ with the properties described in the lemma, and with $t\downto 0$ limit $\vph_0$. We denote the conformal factor of $g_{max}(t)$ by $u(t)$.

We claim that this $g_{max}(t)$ serves as the $g_{max}(t)$ required in  Theorem \ref{DC_max_stretch_thm}.
For this to be true we must establish that for every $\ti g(t)$, $t\in (0,\ti T)$ 
as in the theorem, we have $\ti T\leq T$ and 
$g_{max}(t)\geq \ti g(t)$ for all $t\in (0,\ti T)$.
To see this we will make a second application of Lemma \ref{core_lem}, this time with $\ti T$ there equal to $\ti T$ here, but with $g(t)$ there equal to $\ti g(t)$ here.
The output of the lemma is a flow that we call $\ti g_{max}(t)$, $t\in (0,T)$, 
to distinguish it from the $g_{max}(t)$ already introduced in this proof. 
Similarly, we call the  new potential $\ti\vph$, 
its $t\downto 0$ limit $\ti\vph_0$, and the new conformal factor $\ti u(t)$.

We would like to exploit our knowledge of the potentials $\vph$ and $\ti\vph$ to prove that, in fact, the two Ricci flows $g_{max}(t)$ and $\ti g_{max}(t)$ coincide.
If we achieve that then we will have proved that 
$$g_{max}(t)=\ti g_{max}(t)\geq \ti g(t)$$
for all $t\in (0,\ti T)$, as required.

To establish this equality we first consider the modified potential
$$\psi(t):=\ti \vph(t) +\vph_0-\ti \vph_0 .$$
Note that the static function $\vph_0-\ti \vph_0$ is harmonic 
(not just a weakly harmonic function that is an almost-everywhere representative of a smooth harmonic function)
by the characterisations of $\vph_0$ and $\ti \vph_0$
given by \eqref{potential_evol_eq} of Lemma \ref{potential_construct_lem}.
Consequently $\psi$ satisfies the potential function evolution equation \eqref{potential_evol_eq}. It also satisfies 
$\psi(t)\to\vph_0$ in $L^1_{loc}(M)$ as $t\downto 0$, so the first application of Lemma \ref{core_lem} tells us that $\psi\leq \vph$ throughout $M\times (0,T)$, i.e.
\beq
\label{one_way}
\ti \vph(t) -\vph(t) \leq \ti \vph_0 -\vph_0.
\eeq
On the other hand, we could define a modified potential 
$$\ti \psi(t):=\vph(t) -\vph_0+\ti \vph_0,$$
which would again satisfy \eqref{potential_evol_eq}, and this time satisfy
$\ti \psi(t)\to\ti\vph_0$ in $L^1_{loc}(M)$ as $t\downto 0$. 
We can then use the information from the \emph{second} application of 
Lemma \ref{core_lem} to obtain that $\ti\psi\leq \ti\vph$ throughout $M\times (0,T)$, i.e.
\beq
\label{other_way}
\ti\vph(t)- \vph(t)\geq  \ti \vph_0-\vph_0.
\eeq
Comparing \eqref{one_way} and \eqref{other_way} we find that 
$$\ti\vph(t)- \vph(t)\equiv   \ti \vph_0-\vph_0,$$
so
$$\ti u(t)-u(t)=\lap(\ti\vph(t)-\vph(t))=\lap(\ti\vph_0-\vph_0)=0.$$
We deduce that the conformal factors of $g_{max}(t)$ and
$\ti g_{max}(t)$ are equal so the flows themselves must coincide as required.
\end{proof}

\section{Proofs of the main theorems}

In this section, we prove Theorem \ref{main_exist_thm} and Theorem \ref{ordered_thm}.

\subsection{Proof of the main well-posedness theorem \ref{main_exist_thm}}
\label{main_thm_pf_sect}

Because we are claiming both existence and uniqueness, we may as well lift to the universal cover. In particular, the uniqueness assertion on the universal cover, once proved, will tell us that the flow will descend to the original manifold.
We thus reduce to the three cases that $M$ is $S^2$, the disc $D$  or the plane.

If  $M$ is $S^2$ then our task is  simplified because of the compactness of $S^2$. Indeed, on closed manifolds the existence and uniqueness statements that imply this case have already been proved in the more general K\"ahler setting by Guedj and Zeriahi 
\cite{GZ} and Di Nezza and Lu \cite{DL}.
For convenience, we translate the proof of uniqueness on $S^2$ into the language of this paper in Appendix \ref{sec:s2}.

Suppose then that we are in one of the remaining cases that $M$ is the disc or the plane.
If $T=0$ then $M$ is the plane and $\mu$ is trivial, and then both the uniqueness statement and the vacuous existence statement  of 
Theorem \ref{main_exist_thm} are already included in Theorem \ref{DC_max_stretch_thm}.
Therefore we may assume that $T>0$ from now on.
We obtain the existence of a maximally stretched solution $g_{max}(t)$, for the required length of time, from Theorem \ref{DC_max_stretch_thm}. 
We claim that this will serve as the $g(t)$ required 
in Theorem \ref{main_exist_thm}. 
To establish this it remains to prove uniqueness in these cases. 

For the disc case, let $\tilde{g}(t)$ be the other solution in the assumptions of Theorem \ref{main_exist_thm}. Since both $g(t)$ and $\tilde{g}(t)$ take the same measure $\mu$ as the initial data, we know that for every 
$\eta\in C_c^0(D)$,
\[
	\lim_{t\downto 0} \left(\int_D \eta d\mu_{g(t)} - \int_D \eta d\mu_{\tilde{g}(t)}\right) =0.
\]

For  $R\in (0,1)$, by picking any $\eta\in C_c^0(D,[0,1])$ with $\eta\equiv 1$ on $B_R$ and using 
the fact that $g(t)\geq \tilde{g}(t)$ (by Theorem \ref{DC_max_stretch_thm}), we have
\begin{equation}
	\label{eqn:initial}
	\lim_{\varepsilon\downto 0} \left(\mu_{g(\varepsilon)}(B_R) - \mu_{\tilde{g}(\varepsilon)}(B_R)\right)=0.
\end{equation}

For any $t>0$ for which both $g(t)$ and $\tilde{g}(t)$ are defined, we apply Lemma 3.3 of \cite{ICRF_UNIQ} to the time interval $[\varepsilon,t]$ to find that for $\frac{1}{2}<r_0<r_0^{1/3}<R<1$ and any fixed $\gamma\in (0,1/2)$, 
\begin{equation*}
\begin{aligned}
\left[ \mu_{g(t)}(B_{r_0})- \mu_{\tilde{g}(t)}(B_{r_0})\right]^{\frac{1}{1+\gamma}} \mkern-100mu  \\
 &\leq \left[ \mu_{g(\varepsilon)}(B_R) - \mu_{\tilde{g}(\varepsilon)}(B_R) \right]^{\frac{1}{1+\gamma}} \\
& \quad + C(\gamma) \left[ \frac{t-\varepsilon}{(-\log r_0) [\log (-\log r_0)-\log(-\log R)]^\gamma} \right]^{\frac{1}{1+\gamma}}.
\end{aligned}
\end{equation*}
By \eqref{eqn:initial}, taking $\varepsilon\downto 0$ first and then $R\upto 1$ in the above inequality, we obtain
\[
	\mu_{g(t)}(B_{r_0})=\mu_{\tilde{g}(t)}(B_{r_0})
\]
for any $r_0\in (1/2,1)$. Since $g(t)\geq \tilde{g}(t)$, we deduce that  $g(t)=\tilde{g}(t)$ as long as both of them are defined.

For the $\Real^2$ case, using the same notation as above, for any $R>0$ we still have \eqref{eqn:initial}. With the standard flat metric on $\Real^2$ as background, denote the conformal factors of $g(t)$ and $\tilde{g}(t)$ by $u$ and $\tilde{u}$ respectively. Notice that by the maximality in Theorem \ref{DC_max_stretch_thm}, $\tilde{u}(t)\leq u(t)$.

It is our goal to show $u(\tau)=\tilde{u}(\tau)$ for each $\tau>0$ for which both $\tilde{g}$ and $g$ are defined. For that purpose, we need the following lemma, which is close in spirit to \cite[Theorem 2.1]{ADE1997}, and closer still to the proof of the variant to be found in \cite[Theorem 4.4.1]{giesen_thesis}.

\begin{lem}
\label{initial_R2_lem}
Suppose $g(t)\geq \tilde{g}(t)$ are (conformal) Ricci flows on $\R^2$ for $t\in (0,\tau]$, with conformal factors $u(t)$ and $\tilde{u}(t)$ respectively, such that for some $\ep>0$ we have
\beq
\label{R2_lower_bd}
\tilde{u}(x,t)\geq \frac{\ep t}{(|x|\log|x|)^2}
\eeq
for every  $x\in\R^2\setminus B_2$ and every $t\in (0,\tau]$.
Then for $0<s<t\leq \tau$, $R>2$ and any $m\in (0,1)$ we have
\beqa
\label{initial_R2_lem_est}
\left(\int_{B_R}[u(t)-\tilde{u}(t)] dx\right)^{1-m}
 & \leq \left(\int_{B_{R^2}}[u(s)-\tilde{u}(s)] dx\right)^{1-m}\\
&\quad
+\frac{c_0(m)}{\ep^m (\log R)^{1-m}}[t^{1-m}-s^{1-m}]
\eeqa
for some constant $c_0$ depending only on $m$.
\end{lem}
The proof of this lemma is almost the same as that of Theorem 2.1 of \cite{ADE1997}, except that we have assumed that \eqref{R2_lower_bd} holds uniformly on $(0,\tau]$ so that the inequality \eqref{initial_R2_lem_est} is valid for arbitrarily small $s$. Moreover, since we have assumed $u\geq \tilde{u}$, we do not need to take the positive part as in \cite{ADE1997}.

To verify \eqref{R2_lower_bd} for $t\in (0,\tau]$, we apply Lemma 4.4 of \cite{TY3} to $\tilde{g}(\tau)$ to obtain
\[
	\tilde{u}(x,\tau)\geq \frac{\eta}{(\abs{x}\log \abs{x})^2}
\]
on $\Real^2\setminus B_2$ for some $\eta$ depending on $\tilde{u}$ and $\tau$. We then appeal to the monotonicity of $t\mapsto \tilde{u}(x,t)/t$ in $t$ (see Remark \ref{uovert_mono}) to deduce that
\[
	\tilde{u}(x,t)\geq \frac{\varepsilon t}{(\abs{x}\log \abs{x})^2}
\]
for all $t\in (0,\tau]$ and $\varepsilon=\frac{\eta}{\tau}$.

Finally, we apply Lemma \ref{initial_R2_lem} to see that $u(t)=\tilde{u}(t)$ for any $t\in (0,\tau]$, by taking $s\downto 0$ in \eqref{initial_R2_lem_est} (see \eqref{eqn:initial}) and then $R\to \infty$. This concludes the proof of Theorem \ref{main_exist_thm}.

\subsection{Proof of the ordering theorem \ref{ordered_thm}}
\label{ordered_sect}

In this section we prove Theorem \ref{ordered_thm}. Our first observation is that given the uniqueness result in Theorem \ref{main_exist_thm}, we may as well assume that $M$ is simply connected. If otherwise, we may consider the lift to the universal cover $\tilde{M}$ and the uniqueness implies that the lift of $g(t)$ is {\bf the} unique Ricci flow starting from the lift of $\mu$. If we know it is larger than the lift of $\tilde{g}$, then $g$ is larger than $\tilde{g}$ as we desire.


Secondly, we may assume that $\tilde{g}$ is complete. If $M$ is compact, then this is automatic. If $M$ is the disc or the plane, we apply Theorem \ref{DC_max_stretch_thm} to get a complete Ricci flow $\tilde{g}_{max}(t)$ with $\nu$ as the initial data and satisfying
\[
	\tilde{g}_{max}(t) \geq \tilde{g}
\]
as long as they both exist. 
Note that $\nu$ is forced to be nonatomic, as required by Theorem \ref{DC_max_stretch_thm}, because $\nu\leq \mu$. 
By replacing $\ti g(t)$ with $\tilde{g}_{max}(t)$, we may assume that $\ti g(t)$ is complete as desired.

Recall that for the proof of existence in \cite{TY3}, we constructed a sequence of smooth initial metrics $g_i$ converging to the Radon measure $\mu$ in some sense (see Lemma 4.1 of \cite{TY3}). If $g_i(t)$ is the complete Ricci flow from $g_i$ (Theorem 1.6 of \cite{TY3}), we proved that they converge 
to a complete Ricci flow starting from $\mu$. A priori, this limit flow may depend on the choice of approximations.

However, as a corollary of the uniqueness result in Theorem \ref{main_exist_thm}, we may now take any approximating sequence $g_i$ and the argument in \cite{TY3} gives the unique complete Ricci flow solution starting with the given initial measure. For the purpose of proving Theorem \ref{ordered_thm}, it therefore suffices to find two sequences of approximating smooth initial metrics for $\mu$ and $\nu$ respectively such that the order is preserved, because their limits, which by the uniqueness result are $g(t)$ and $\tilde{g}(t)$, will be ordered.

The following mollification result is similar to \cite[Lemma 4.1]{TY3}. It is slightly simpler since we can now work on the universal cover. On the other hand we claim a little more in the sense that we need to compare the smoothings of two measures that are known to be ordered.
\begin{lem}[Smoothing lemma]
Suppose the Riemann surface $M$ is either the disc, the plane or the sphere.
We can find a map from the space of Radon measures $\mu$ on $M$ and numbers $h>0$
to the space of conformal Riemannian metrics $g_h(\mu)$ on $M$ of finite total volume, 
with the properties that for every fixed Radon measure $\nu$, we have
$$\mu_{g_h(\nu)}\weakto \nu\qquad\text{ and }\qquad
\mu_{g_h(\nu)}(M)\to \nu(M)$$
as $h\downto 0$, and so that for all Radon measures $\nu_1$ and $\nu_2$ on $M$ with 
$\nu_1\leq \nu_2$, and for every $h>0$, we have
$$g_h(\nu_1)\leq g_h(\nu_2).$$
\end{lem}

\begin{proof}
Consider first the case that $M$ is the plane. 
Let $g_0$ be a metric on $M$ corresponding to the round punctured unit sphere, pulled back by stereographic projection.
Given a Radon measure $\nu$ on the plane, and $h>0$, we can restrict $\nu$ to the ball $B_{1/h}$
and then mollify $\nu$ in the traditional way over the whole plane (with respect to the scale parameter $h$)
to give a smooth conformal factor of a degenerate Riemannian metric on $\R^2$, with compact support. If we add on $hg_0$ then we obtain a nondegenerate Riemannian metric $g_{h}(\nu)$ with the required properties.

The case that $M$ is the disc can be handled in a similar manner. Given a Radon measure $\nu$ on $D$ and $h\in (0,\half)$,
we restrict to $B_{1-2h}$, mollify and add $hg_1$, where $g_1$ is (say) the Lebesgue measure of the unit disc.



The case that $M=S^2$ will follow instantly from mollification alone, although we should use mollification with respect to the distance of the round spherical metric.
\end{proof}

\appendix

\section{The Poisson equation}

It will be important to be able to solve Poisson's equation with smooth, but otherwise uncontrolled, inhomogeneous term $f$.
\begin{lem}
	\label{lem:poisson}For any smooth function $f$ on the unit disc $D$ or the plane $\Real^2$, there exists a smooth solution $u$ to the Poisson equation
\begin{equation}
	\Delta u =f.
\end{equation}
\end{lem}
\begin{proof}
	We give a proof for the $\Real^2$ case only. The other case can be proved by similar arguments.

	For any subset $E$ of the domain, denote the characteristic function of $E$ by $\chi_E$. Let $f_1=f\chi_{B_1}$ and
	\[
		f_i= f \cdot \chi_{B_i\setminus B_{i-1}}
	\]
	for $i>1$. Obviously, $f=\sum_{i=1}^\infty f_i$, and 
 for each $i\in \N$ we have a solution $u_i$ (given by a convolution for example) to the equation
	\[
		\Delta u_i=f_i \qquad \text{on} \quad \Real^2.
	\]
	By the definition of $f_i$, we know $u_i$ is harmonic on the disc $B_{i-\frac{5}{4}}$ for $i>1$. Hence, it is the uniform limit of some converging power series there. Moreover, the partial sums are harmonic polynomials. By subtracting from $u_i$ the partial sum of sufficiently many terms, we may get another solution $v_i$ to the Poisson equation $\Delta v_i=f_i$ on $\R^2$, satisfying
	\begin{equation}
		\label{eqn:goodvi}
		\norm{v_i}_{C^0(B_{i-\frac{3}{2}})}\leq 2^{-i}.
	\end{equation}
	Set $v_1=u_1$. By \eqref{eqn:goodvi}, the series $\sum_{i=1}^\infty v_i$ converges to some function $v$ uniformly on any compact set of $\Real^2$ and gives a classical solution to $\Delta v= f$.
\end{proof}
The idea of the proof comes from Theorem 4 in Chapter 5 of Ahlfors' book \cite{Ahlfors}.

\section{A semi-linear equation}


In this section, we study the equation
\begin{equation}\label{eqn:mykw}
	\Delta_0 w = e^{w- f} -1
\end{equation}
on $S^2$, where $\Delta_0$ stands for the Laplacian of the round metric $g_0$ on $S^2$ and $f$ is some given function. It is closely related to the problem of prescribing Gauss curvature on surfaces.  


\begin{lem}
	\label{lem:kw}
For each smooth $f:S^2\to (-\infty,0]$,
there exists a solution $w\in C^\infty(S^2)$ 
%
to \eqref{eqn:mykw} such that for any $p\in (1,\infty)$
we have 
\[
	\norm{w}_{C^0(S^2)}\leq C
\]
for some constant $C$ depending only on $p$ and an upper bound for $\norm{e^{-f}}_{L^p}$.
\end{lem}

The proof below is adapted from Section 9 and 10 in \cite{KW}. It was shown there that we can use the method of upper and lower solutions to solve
\begin{equation}
	\label{eqn:kw}
	\Delta_0 u = c- he^u
\end{equation}
where $c$ is a constant and $h$ is a smooth function. Notice when $c=-1$ and $h=-e^{-f}$, \eqref{eqn:kw} becomes \eqref{eqn:mykw}.

\cmt{Kazdan and Warner attribute this to Courant and Hilbert vol 2, p. 370-371 (which I don't have here)}

\begin{lem}[Lemma 9.3 in \cite{KW}]
	Let $c<0$ be a constant and $h\in C^\infty(S^2)$ be a smooth function. If there exist smooth upper and lower solutions $u_+$ and $u_-$ in the sense that
 \begin{equation}
 \label{eqn:upperlower}
		\Delta_0 u_- -c+he^{u_-}\geq 0; \quad \Delta_0 u_+ -c +he^{u_+}\leq 0
 \end{equation}
	and if $u_-\leq u_+$, then there is a smooth solution $u$ to \eqref{eqn:kw} satisfying $u_-\leq u\leq u_+$.
\end{lem}

Lemma \ref{lem:kw} follows from the above lemma if we can construct some upper and lower solutions  that  are bounded by a constant depending only on $p$ and 
an upper bound for
the $L^p$ norm of $h$. The following construction is from \cite{KW}. However, it is simplified since 
we have assumed that $f\leq 0$, i.e. $h\leq -1$.

\begin{lem}[From Lemma 9.5 and Theorem 10.5 (a) in \cite{KW}]
For the case $c=-1$ and $h\leq -1$, we can construct smooth functions $u_{+}$ and $u_-$ satisfying \eqref{eqn:upperlower} and
\[
	-C\leq u_-\leq u_+\leq C,
\]
for some $C$ depending only on $p$ and an upper bound for $\norm{h}_{L^p}$.
\end{lem}

\begin{proof}
	Let $\bar{h}$ be the average of $h$ with respect to $g_0$, which by our assumption is no larger than $-1$. In the following proof, the constant $C$ varies from line to line and depends only on 
 $p$ and an upper bound for  $\norm{h}_{L^p}$. Solve
	\[
		\Delta_0 v = \bar{h}-h,
	\]
	with the normalization that the average of $v$ vanishes, i.e. $\bar{v}=0$. Hence, the usual $L^p$ estimate implies that
	\[
		\norm{v}_{C^0(S^2)}\leq C.
	\]
	Set $u_+=v+\norm{v}_{C^0}$ and we find that
\[
	\Delta_0 u_+ -c + he^{u_+} = \bar{h}-h + 1 + h e^{v+ \norm{v}_{C^0}}\leq 0.
\]
By setting $\alpha= (- \bar{h})^{-1}$, we know 
$\al\geq 1/C$.
Solve the Poisson equation
\[
	\Delta_0 w = \alpha (-h) -1
\]
with 
$\bar{w}=0$.
Since $\alpha\leq 1$, we have $\norm{w}_{C^0}\leq C$.
Finally, we take 
$$\lambda:=\norm{w}_{C^0}- \log \alpha\leq C$$ 
and set
\[
	u_-=w-\lambda.
\]
We check that
\[
	\Delta_0 u_- - c + he^{u_-}= \alpha (-h) -1 + 1 + h e^{w -\norm{w}_{C^0}+ \log \alpha}\geq 0.
\]
So far we have constructed an upper solution $u_+$ and a lower solution $u_-$, however, it is not clear if $u_-\leq u_+$. Since $u_+$ and $u_-$ are both bounded by $C$, we have
\[
	u_- -2C \leq u_+.
\]
The key observation is that since $h<0$, $u_--2C$ is also a lower solution.
\end{proof}

\section{Uniqueness on the 2-sphere}
\label{sec:s2}

In this section, we give a self-contained proof of the uniqueness part of Theorem \ref{main_exist_thm} in the case that the universal cover $\tilde{M}$ is $S^2$. 
Some aspects of the proof are easier than in the proof of Theorem \ref{main_exist_thm} owing to 
the compactness of $S^2$. 
In particular, it is straightforward to directly construct a potential corresponding to any Ricci flow.
The arguments used here are adapted from \cite{DL}.

The proof consists of two parts. The first part proves the existence of the potential flow and translates the uniqueness problem into that of the potential flow. The second part uses various maximum principles to prove the uniqueness.

By passing to the universal cover, we may assume that $M=S^2$.
Recall that if $g(t)$ is a conformal Ricci flow solution on $S^2$ then the area decreases linearly because
\beq
\label{8pi_dec}
\frac{d}{dt}\mu_{g(t)}(S^2)=-2\int_{S^2}K_{g(t)}d\mu_{g(t)}=-8\pi
\eeq
by Gauss-Bonnet \cite[(2.5.8)]{RFnotes}.
Suppose in addition that  $g(t)$ attains initial data $\mu$ in the sense that
\[
	\int_{S^2} \eta d\mu = \lim_{t\downto 0} \int_{S^2} \eta d\mu_{g(t)}
\]
for any continuous function $\eta$ on $S^2$. 
By setting $\eta\equiv 1$, we thus deduce that $\mu_{g(t)}(S^2)=\mu(S^2)-8\pi t$. In particular, in the case that $\mu$ is the trivial measure then not only is the existence trivial (because $T=0$) but also, no other solution can exist, giving uniqueness.
We may therefore assume that $\mu$ is not trivial, and $T>0$.
By scaling, we may also assume that $\mu(S^2)=4\pi$.

The existence of at least one solution for the required time is known from \cite{TY3}, cf. \cite{GZ}.

Recall that $g_0$ is the round metric of curvature $1$ and $\Delta_0$ is the Laplacian of $g_0$. If we let $u(t)$ be the conformal factor of $g(t)$ with respect to $g_0$, then the Ricci flow equation 
\eqref{RF_eq}
becomes
\beq
\label{uS2_eq}
	\partial_t u = \Delta_0 \log u -2.
\eeq
Since $g(t)$  attains the measure $\mu$ as initial data, we have the following estimates for $u(t)$:
\begin{enumerate}[(i)]
\item
As justified above, 
the area of $g(t)$ is given by 
\begin{equation}
	\label{eqn:area}
	\int_{S^2}u(t) d\mu_{g_0} =4\pi -8\pi t.
\end{equation}
\item
For  $\ep:=\min_{S^2} \frac{u(1/4)}{1/4}$, we have
\begin{equation}
		\label{eqn:uct}
	u(x, t)\geq \ep t,\qquad \text{for all} \quad (x,t)\in S^2\times (0,1/4].
\end{equation}
This follows from the monotonicity of $\frac{u}{t}$ 
given in Remark \ref{uovert_mono}.
\end{enumerate}


\subsection{The flow of the potential}

The main result of this subsection is the existence of a potential flow.
\begin{lem}
	\label{lem:potentialflow}
	Given $\mu$, $g(t)$ and $u(t)$ as above, there exists a smooth function $\varphi:S^2\times (0,1/2)\to \Real$ such that

\begin{enumerate}[(1)]
\item
	\begin{equation}
		\label{eqn:uphi}
		\Delta_0 \varphi = u - (1-2t).
	\end{equation}
\item
	\begin{equation}
		\label{eqn:phievolve}
		\partial_t \varphi = \log (\Delta_0 \varphi +1-2t).
	\end{equation}
\item
The pointwise limit $\lim_{t\downto 0} \varphi$, denoted by $\varphi_0$, exists in $\Real\cup \set{-\infty}$. Moreover, $\varphi_0$ is an upper semi-continuous function satisfying
	\[
		\Delta_0 \varphi_0 = \mu -1
	\]
	in the sense of distributions. 

\item
$\varphi(t)$ converges to $\varphi_0$ in $L^1$ as $t\downto 0$.

\item
For any $x\in S^2$, we have
    \begin{equation}
        \label{eqn:characterize}
        \varphi_0(x)= \lim_{r\downto 0} \fint_{B_r(x)} \varphi_0\, d\mu_{g_0}.
    \end{equation}
    It is possible that both sides of the above equation are minus infinity.

\cmt{
The notation $B_r$ could mean geodesic ball or stereographic ball.}

\item
By setting $\varphi(0)=\varphi_0$, $\varphi$ as a function on $S^2\times [0,1/2)$ is upper semi-continuous.
\item
For any $p>1$, there exist $\delta>0$ and $C>0$ such that
	\[
		\int_{S^2} e^{p \abs{\varphi(t)}} d\mu_{g_0} \leq C, \qquad \forall\, 0<t<\delta.
	\]
\end{enumerate}
\end{lem}

The potential $\vph$  is unique up to the addition of a time-independent constant.


\begin{proof}
Solve the Poisson equation
\[
	\Delta_0 \varphi(1/4)= u(1/4) - 1/2
\]
and set, for $t\in (0,1/2)$,
\begin{equation}
    \label{eqn:s2varphi}
	\varphi(t) := \varphi(1/4) + \int_{1/4}^t \log u(s) ds.
\end{equation}
%
To see \eqref{eqn:uphi}, we mimic \eqref{vph_t_eq} and compute
\begin{eqnarray*}
	\Delta_0 \varphi(t) &=&\Delta_0 \varphi(1/4) + \int_{1/4}^t \Delta_0 \log u(s) ds \\
	&=& u(1/4)-1/2	+ \int_{1/4}^t (\partial_t u +2 ) ds\\
	&=&u(t)-1+2t.
\end{eqnarray*}
Differentiating \eqref{eqn:s2varphi} in time gives
$\partial_t \varphi = \log u$, which then implies
\eqref{eqn:phievolve}.

The claims in Parts (3)-(6) are either pointwise or local properties so  we may prove 
them in a local coordinate system. Fix $x_0\in S^2$ and 
pick local isothermal coordinates near $x_0$, scaled to exist in an open set containing $\overline{B}$.
Pick $w\in C^\infty(\overline{B})$ satisfying
\[
\Delta_0 w = 1 \qquad \text{on} \quad B,
\]
and set
\[
\tilde{\varphi}(t) = \varphi(t) + (1-2t) w.
\]
Hence,
\[
\Delta \tilde{\varphi}(t) = u_0 u(t) \geq 0
\]

where $u_0$ is the conformal factor of $g_0$, i.e. 
$d\mu_{g_0}=u_0 dx$
and 
setting
\[
\tilde{u}= u \cdot e^{-2w},
\]
\eqref{eqn:s2varphi} becomes
\[
\tilde{\varphi}(t)= \varphi(1/4)+ \frac{w}{2} + \int_{1/4}^t \log \tilde{u}(s) ds.
\]
In particular, we have $\tilde{\varphi}(1/4)=\varphi(1/4)+\frac{w}{2}$.
Keeping in mind \eqref{eqn:area} and \eqref{eqn:uct},
we may apply Lemma \ref{potential_sublem} with $T=1/2$, $\tau=1/4$, $R=1$ and with $\tilde{u}$ and $\tilde{\varphi}$ in place of $u$ and $\varphi$ therein. 
%
The local properties of $\tilde{\varphi}$ proved in Lemma \ref{potential_sublem} imply the claims for $\varphi$ in Parts (3)-(5) except
\[
\Delta_0 \varphi_0 =\mu -1. 
\]
To see this, we take any continuous function $\eta$ and compute
\begin{align*}
    \int_{S^2} \eta \Delta_0 \varphi_0 d\mu_{g_0} &= \lim_{t\downto 0} \int_{S^2} \Delta_0 \eta \varphi(t) d\mu_{g_0} \\
    &= \lim_{t\downto 0} \int_{S^2} \eta (u(t)-1+2t) d\mu_{g_0} \\
    &= \int_{S^2} \eta d\mu - \int_{S^2} \eta d\mu_{g_0}.
\end{align*}
For the proof of Part (6), we apply Lemma \ref{upper_semi_lem} to $\tilde{\varphi}$ and notice that the assumptions there are proved by our previous application of Lemma \ref{potential_sublem}.



For Part (7), we first find a cover of $S^2$ by balls $B^1, \cdots, B^m$, each of geodesic radius $r<\pi/8$, such that
\[
	\int_{B^j(4)} d\mu < \frac{\pi}{p},\qquad j=1,2,\cdots,m.
\]
Here $B^j(m)$ is the ball of the same center as $B^j$, but $m$ times the radius.
By Lemma 3.2 of \cite{TY3}, there is $\delta>0$ such that for any $t\in (0,\delta)$,
\[
	\int_{B^j(2)} u(t) d\mu_{g_0} \leq \frac{2\pi}{p}.
\]
We may assume that the radius of $B_j$ are small so that
\[
	\int_{B^j(2)} \abs{\Delta_0 \varphi(t)} d\mu_{g_0} \leq \frac{3\pi}{p}.
\]
Now, we apply a result of Brezis-Merle \cite{BM} to conclude that
\footnote{
	Here is how we apply Theorem 1 of \cite{BM}. Let $v$ be the harmonic function on $B^j(2)$ with boundary value $v|_{\partial B^j(2)}= \varphi(t)$. For small $\eta\in (0,\pi)$ satisfying
	\[
		\frac{(4\pi-\eta)}{\norm{\Delta \varphi(t)}_{L^1(B^j(2))}}>p,
	\]
 we have
\begin{equation}
	\label{eqn:phiv}
\int_{B^j(2)} \exp (p \abs{\varphi(t)-v}) d\mu_0\leq C(\delta, j).
\end{equation}
By Jensen's inequality, $\norm{\varphi(t)-v}_{L^1(B^j(2))}\leq C$. By (4) again, we know $\norm{v}_{L^1(B^j(2))}\leq C$. Since $v$ is a harmonic function, we obtain the $L^\infty$ of $v$ on $B^j$ and the desired inequality follows from \eqref{eqn:phiv}.
}
\[
	\int_{B^j} e^{p \abs{\varphi(t)}} d\mu_{g_0}\leq C.
\]
Adding these together, we get Part (7).
\end{proof}

\subsection{Comparison}

Given two solutions $u_1$ and $u_2$ of \eqref{uS2_eq} satisfying the same initial data $\mu$ with $\mu(S^2)=4\pi$, we apply Lemma \ref{lem:potentialflow} to get $\varphi_1$ and $\varphi_2$ respectively. By Parts (3) and (5) in that lemma, $\varphi_1- \varphi_2$ is a harmonic function on $S^2$, hence a constant. Moreover, by the upper semi-continuity, $\varphi_1$ and $\varphi_2$ are bounded from above. Therefore, we may assume by adding some constants that
\[
	\varphi_1(0)=\varphi_2(0)<0.
\]

Our goal is to show that $\vph_1(t)\equiv \vph_2(t)$ since that will give our desired conclusion that 
$u_1\equiv u_2$.

\begin{lem}
\label{lem:c2}
	For any $t\in (0,1/2)$, $\varphi_i(t)<0$ for $i=1,2$.
\end{lem}

\begin{proof}
    By Part (6) of Lemma \ref{lem:potentialflow} and the assumption that $\varphi_i(0)<0$, we can find small $\delta>0$ such that
    \[
    \varphi_i(t)<0\qquad \forall t\in (0,\delta].
    \]
    If the lemma is not true (for some $i$), then we may find some $(x_0,t_0)\in S^2\times (0,1/2)$ satisfying
    \[
    \varphi_i(x_0,t_0)=0
    \]
    and
    \[
    \varphi_i(x,t)<0, \qquad \forall x\in S^2, \, t\in (0,t_0).
    \]
    Since $\varphi_i$ is smooth for $t>0$, we have $\partial_t \varphi_i (x_0,t_0)\geq 0$ and $\Delta_0 \varphi_i(x_0,t_0)\leq 0$, which is a contradiction to the equation
	\[
		\partial_t \varphi_i = \log(\Delta_0 \varphi_i +1-2t).
	\]
\end{proof}

If both $u_1$ and $u_2$ are smooth on $S^2\times [0,1/2)$, then by the well-known uniqueness of smooth Ricci flow on compact manifolds, we always have $u_1=u_2$. Or, one may apply the classical maximum principle to \eqref{eqn:phievolve} to see that $\varphi_1=\varphi_2$ on $S^2\times [0,1/2)$.

When the initial data is only a measure, we know from Lemma \ref{lem:potentialflow} that $\varphi_i(0)$ may not be bounded from below so that the classical maximum principle argument will not work. We use the same argument as in Di Nezza and Lu to solve this problem.

\begin{lem}[Lemma 5.1 in \cite{DL}]
	\label{lem:deeper}
	Let $\varphi$ be $\varphi_1$ or $\varphi_2$. We have
	\begin{equation}
		\label{eqn:deeper}
		\varphi(t)\geq (1-4t)\varphi(0)-\alpha(t), \qquad \forall\, t\in (0,1/4)
	\end{equation}
where $\alpha(t)$ is a positive function satisfying $\lim_{t\downto 0}\alpha(t)=0$.	
\end{lem}


\begin{proof}
Fix some $p>1$. For any $\varepsilon>0$ smaller than $\min (\delta,1/8)$ for the $\delta$ in Part (7) of Lemma \ref{lem:potentialflow},  we have some $C>0$ independent of $\varepsilon$ such that
\begin{equation}
    \label{eqn:uniformlp}
\norm{e^{-4\varphi(\varepsilon)}}_{L^p}\leq C.
\end{equation}
By Lemma \ref{lem:c2}, we also have $4\varphi(\varepsilon)<0$.
 Hence, by setting $f=4\varphi(\varepsilon)$,
 Lemma \ref{lem:kw} implies the existence of $w_\varepsilon$ satisfying
\[
	\Delta_0 w_\varepsilon = e^{w_\varepsilon -4\varphi(\varepsilon)}-1
\]
with
\begin{equation}
    \label{eqn:uniformc0}
    \norm{w_\varepsilon}_{C^0(S^2)}\leq C'
\end{equation}
for some $C'$ depending on $p$ and the $C$ in \eqref{eqn:uniformlp}. 

Define $t\in (0,1/4)$, 
\begin{equation}
	\label{eqn:aux}
	\Phi_\varepsilon(t)= (1-4t) \varphi(\varepsilon) + t w_\varepsilon+ (t\log t- t).
\end{equation}
By direct computation, we have
\begin{eqnarray*}
	\partial_t \Phi_\varepsilon &=&   -4\varphi(\varepsilon) + w_\varepsilon + \log t \\
	&=& \log (t \Delta_0 w_\varepsilon + t)
\end{eqnarray*}
and (using that $t\in (0,1/4)$)
\begin{eqnarray*}
	&&\Delta_0 \Phi_\varepsilon -2 (t+ \varepsilon)+ 1 \\
 &=&  (1-4t) \Delta_0 \varphi(\varepsilon) + t \Delta_0 w_\varepsilon +1 -2(t+\varepsilon)\\
	&=& (1-4t)(\Delta_0 \varphi(\varepsilon)+1-2\varepsilon) + t\Delta_0 w_\varepsilon + (1-2t-2\varepsilon)- (1-2\varepsilon)(1-4t)\\
	&\geq&  t\Delta_0 w_\varepsilon + t.
\end{eqnarray*}
Here in the last line above, we used \eqref{eqn:uphi} and the fact that $\varepsilon<1/8$ so that
\[
	(1-2t-2\varepsilon)-(1-2\varepsilon)(1-4t)=2t-8\varepsilon t\geq t.
\]
Hence,
\[
	\partial_t \Phi_\varepsilon \leq \log (\Delta_0 \Phi_\varepsilon - 2 (t+\varepsilon) +1).
\]
If we set $\Psi_\varepsilon(t)=\varphi(\varepsilon+t)$, then \eqref{eqn:phievolve} implies that
\[
	\partial_t \Psi_\varepsilon = \log (\Delta_0 \Psi_\varepsilon - 2 (t+\varepsilon) +1).
\]
Since $\Phi_\varepsilon(0)= \varphi(\varepsilon)=\Psi_\varepsilon(0)$, we apply the usual maximum principle to conclude
\[
	\Phi_\varepsilon(t)\leq \Psi_\varepsilon(t)=\varphi(\varepsilon+t)
\]
for $t\in [0, 1/4-\varepsilon)$. 

\cmt{everywhere else in the paper we were a lot more careful about the maximum principle in an unfamiliar setting like this. But I am happy leaving it like this because by now the reader would know that one could actually define
$\Psi_\varepsilon(t)=(1+t)\eta+\varphi(\varepsilon+t)$ for small $\eta>0$, then do the bare hands maximum principle, then let $\eta\downto 0$.}

For fixed $0<t<1/4$, taking $\varepsilon\to 0$ and using \eqref{eqn:uniformc0}, we obtain
\[
	\varphi(t)\geq (1-4t)\varphi(0) - \alpha(t).
\]
\end{proof}

We continue with our objective of proving that $\vph_1(t)\equiv \vph_2(t)$.

Instead of comparing $\varphi_1$ and $\varphi_2$, we set
\[
	\psi_i(t)= e^t \varphi_i(\frac{1-e^{-t}}{2}).
\]
Since $\varphi_i$ is defined on $[0,\frac{1}{2})$, $\psi_i$ is defined on $[0,\infty)$.
To compute the evolution of $\psi_i$, 
\begin{equation}
	\label{eqn:psi2}
\begin{split}
	\partial_t \psi_i &= \psi_i + \frac{1}{2} \partial_t \varphi_i(\frac{1-e^{-t}}{2})\\
	&= \psi_i +\frac{1}{2}\log \left(\Delta_0 \varphi_i (\frac{1-e^{-t}}{2}) +1 - 2(\frac{1-e^{-t}}{2})\right)\\
	&= \psi_i +\frac{1}{2}\log \left(\Delta_0 \varphi_i (\frac{1-e^{-t}}{2}) + e^{-t}\right)\\
	&= \psi_i - \frac{t}{2}+\frac{1}{2}\log \left(\Delta_0 \psi_i(t) + 1\right).
\end{split}	
\end{equation}

Let $\alpha(t)$ be the function given in Lemma \ref{lem:deeper}. For  $\varepsilon>0$ small enough so that
\beq\label{ep_constraints}
\textstyle \frac{1}{2}(1-e^{-\varepsilon})\in (0,1/4)\qquad\text{and}\qquad
2-e^\varepsilon\in (0,1),
\eeq
we set
\[
	\xi(t)= \psi_1(t+\varepsilon) + \frac{1}{2}\varepsilon(e^t-1) + e^{t+\varepsilon} \alpha(\frac{1-e^{-\varepsilon}}{2}),
\]
defined for $t>0$, and compute
\begin{equation*}
\begin{aligned}
\partial_t \xi &= \partial_t \psi_1(t+\varepsilon) + \frac{1}{2}\varepsilon e^t + e^{t+\varepsilon} \alpha(\frac{1-e^{-\varepsilon}}{2}) \\
&= \psi_1(t+\varepsilon) -\frac{(t+\varepsilon)}2 +\frac{1}{2}\log(\Delta_0 \psi_1(t+\varepsilon)+1) + \frac{1}{2}\varepsilon e^t + e^{t+\varepsilon} \alpha(\frac{1-e^{-\varepsilon}}{2}) \\
&= \xi(t) -\frac{t}{2} + \frac{1}{2}\log (\Delta_0 \xi(t) +1).
\end{aligned}
\end{equation*}
Hence, $\xi$ is a classical solution to \eqref{eqn:psi2} and
\[
	\xi(0)=\psi_1(\varepsilon)+ e^\varepsilon \alpha(\frac{1-e^{-\varepsilon}}{2})= e^\varepsilon \varphi_1 (\frac{1-e^{-\varepsilon}}{2}) + e^{\varepsilon} \alpha(\frac{1-e^{-\varepsilon}}{2}).
\]
Applying Lemma \ref{lem:deeper} with $t=\frac{1}{2}(1-e^{-\varepsilon})$, 
keeping in mind the first part of \eqref{ep_constraints}, we have
\[
	\varphi_1(\frac{1-e^{-\varepsilon}}{2}) \geq \left(1-4(\frac{1-e^{-\varepsilon}}{2})\right) \varphi_1(0) - \alpha(\frac{1-e^{-\varepsilon}}{2}).
\]
Hence,
\begin{equation}
    \label{eqn:xiphi2}
	\xi(0) \geq e^\varepsilon \left(1-4(\frac{1-e^{-\varepsilon}}{2})\right) \varphi_1(0) > \varphi_1(0)=\psi_2(0).
\end{equation}
Here in the last line above, we used 
$e^{\varepsilon}(1-4(\frac{1-e^{-\varepsilon}}{2}))=2-e^\varepsilon \in (0,1)$,
by \eqref{ep_constraints}, and the fact that $\varphi_1(0)<0$.

Since $\xi$ and $\psi_2$ satisfy the same equation with the initial values satisfying \eqref{eqn:xiphi2}, an application of the maximum principle will imply that for all $x\in S^2$ and $t>0$,
\begin{equation}
    \label{eqn:maximum}
\xi(x,t)\geq \psi_2(x,t).
\end{equation}
If otherwise, we have $(x_1,t_1)$ satisfying
\[
\xi(x_1,t_1)< \psi_2(x_1,t_1).
\]
Pick a small $\delta>0$ such that
\[
\xi(x_1,t_1)+ \delta (e^{t_1}-1)< \psi_2(x_1,t_1).
\]
Now, we claim the existence of $(x_0,t_0)\in S^2\times (0,t_1]$ such that
\[
\xi(x_0,t_0)+ \delta (e^{t_0}-1)= \psi_2(x_0,t_0)
\]
and for all $x\in S^2$ and $t\in (0,t_0)$,
\[
\xi(x,t)+ \delta (e^t-1) > \psi_2(x,t).
\]
If the claim is not true, then we have a sequence of $(y_i,s_i)$ satisfying $y_i\to y_0$, $s_i\downto 0$ and
\[
\xi(y_i,s_i)+ \delta (e^{s_i}-1)\leq \psi_2(y_i,s_i).
\]
By the smoothness of $\xi$ at $t=0$ and the upper semi-continuity of $\psi_2$ (a property inherited from $\varphi_2$), we take the limit $i\to \infty$ in the above inequality to get a contradiction to \eqref{eqn:xiphi2}, which proves the claim.

The existence of such $(x_0,t_0)$ implies
\[
\partial_t \xi (x_0,t_0) + \delta e^{t_0} \leq \partial_t \psi_2(x_0,t_0)
\]
and
\[
\Delta_0 \xi(x_0,t_0) \geq \Delta_0 \psi_2(x_0,t_0).
\]
These inequalities contradict the equation satisfied by both $\xi$ and $\psi_2$ (see \eqref{eqn:psi2}) and prove \eqref{eqn:maximum}.


Finally, for any fixed $t>0$, taking $\varepsilon$ to $0$ in \eqref{eqn:maximum}, we obtain
\[
	\psi_1\geq \psi_2,\qquad \text{on} \quad S^2\times[0,1/8).	
\]
By symmetry, we have $\psi_1=\psi_2$, hence $\varphi_1=\varphi_2$.



\noindent
\url{https://homepages.warwick.ac.uk/~maseq/}\\
\noindent
{\sc Mathematics Institute, University of Warwick, Coventry,
CV4 7AL, UK.} 

{\sc School of Mathematical Sciences, University of Science and Technology of China, Hefei, 230026, P.R.China} 

\end{document}